\documentclass{conm-p-l}
\usepackage{mathrsfs}
\usepackage{amsmath, amsthm, amssymb, enumerate, mathtools}

\usepackage{color}

\numberwithin{equation}{section}

\title[K\"ahler surfaces]{Sca\-lar-flat K\"ah\-ler surfaces whose Weyl\\
tensor annihilates the Ric\-ci form}

\author[A.\,Derdzinski, S.\,Kim, J.-h.\,Park]{Andrzej Derdzinski$^{1}$,
Sinhwi Kim$^{2}$, and JeongHyeong Park$^{2}$}
\address{$^{1}$\nnh\ Department\nnh\ of\nnh\ Mathematics,\nnh\ The\nnh\
Ohio\nnh\ State\nnh\ University,\nnh\ Columbus,\nnh\ OH\nnh\ 43210,\nnh\ USA}

\email{andrzej@math.ohio-state.edu}

\address{$^{2}$ Department of Mathematics, Sungkyunkwan University, Suwon, 16419, Korea}
\email{kimsinhwi@skku.edu, parkj@skku.edu}


\def\bbC{{\mathchoice {\setbox0=\hbox{$\displaystyle\rm C$}\hbox{\hbox
to0pt{\kern0.4\wd0\vrule height0.9\ht0\hss}\box0}}
{\setbox0=\hbox{$\textstyle\rm C$}\hbox{\hbox
to0pt{\kern0.4\wd0\vrule height0.9\ht0\hss}\box0}}
{\setbox0=\hbox{$\scriptstyle\rm C$}\hbox{\hbox
to0pt{\kern0.4\wd0\vrule height0.9\ht0\hss}\box0}}
{\setbox0=\hbox{$\scriptscriptstyle\rm C$}\hbox{\hbox
to0pt{\kern0.4\wd0\vrule height0.9\ht0\hss}\box0}}}}

\def\txm{{T\hskip-2.9pt_x\w\hn M}}
\def\tam{{T\hskip-.2pt^*\hskip-2.3ptM}}

\newcommand{\sca}{\mathrm{s}}
\newcommand{\w}{^{\phantom i}}
\newcommand{\nnh}{\hskip-1pt}
\newcommand{\nnn}{\hskip-2.5pt}

\def\aj{A}
\def\bj{B}
\def\cj{C}
\def\dj{D}
\def\ej{E}
\def\fj{F}
\def\fe{\fj\hskip-2.3pt_1\w}
\def\fz{\fj\hskip-2.3pt_2\w}
\def\fd{\fj\hskip-2.3pt_3\w}
\def\fv{\fj\hskip-2.3pt_4\w}
\def\gj{G}
\def\hj{H}

\def\lj{L}

\def\sj{S}

\def\ly{\lambda}
\def\sy{\sigma}

\def\ry{\rho}

\usepackage{color}

\newcommand{\ric}{\mathrm{r}}
\newcommand{\sym}{\mathrm{b}}

\def\cwedge{\bigcirc\kern-1.07em\wedge\ }
\newcommand{\hyp}{\hskip.5pt\vbox
{\hbox{\vrule width2.5ptheight0.5ptdepth0pt}\vskip2pt}\hskip.5pt}
\def\vt{{\tau\hskip-4.55pt\iota\hskip.6pt}}
\def\hs{\hskip.7pt}
\def\hh{\hskip.4pt}
\def\nh{\hskip-.7pt}
\def\hn{\hskip-.4pt}
\def\w{^{\phantom i}}

\def\vg{\varGamma}

\newtheorem{thm}{Theorem}[section]
\newtheorem{lem}[thm]{Lemma}

\theoremstyle{definition}

\theoremstyle{remark}
\newtheorem{rem}[thm]{Remark}

\numberwithin{equation}{section}

\begin{document}
\begin{abstract}
We conjecture that any scalar-flat K\"ahler surface
in which the Weyl tensor acting on $\,2$-forms annihilates the Ric\-ci form
must be either Ric\-ci-flat or locally isometric to  
a Riemannian product of two real surfaces with mutually 
opposite nonzero constant Gauss\-i\-an curvatures. This amounts to the
nonexistence of proper weakly Ein\-stein anti-self-dual K\"ahler surfaces.
We prove the above conjecture in three special cases: when the manifold
is compact, when 
one of the Ric\-ci eigen\-dis\-tri\-bu\-tions is in\-te\-gra\-ble, and when 
the norms of the Ric\-ci and Weyl tensors are functionally dependent.
\end{abstract}

\maketitle

\setcounter{thm}{0}
\renewcommand{\thethm}{\Alph{thm}}
\section{Introduction}
Following Euh {\it et al.}\ \cite[p.\,112]{euh-park-sekigawa-rm} we call a
Riemannian
four-man\-i\-fold {\it weakly Ein\-stein\/} if the triple contraction of its
curvature tensor against itself is a functional multiple of the metric, and
then refer to it as {\it proper\/} \cite{derdzinski-park-shin-ct} if it is
neither Ein\-stein, nor con\-for\-mal\-ly flat and scalar-flat. As shown by
Euh {\it et al.}\ \cite[Cor.\,2.2]{euh-park-sekigawa-rm} and
Garc\'\i a-R\'\i o {\it et al.}\
\cite[Thm.\,2(i)]{garcia-rio-haji-badali-marino-villar-vazquez-abal}, 
the weakly Ein\-stein property follows in both cases excluded by the 
word {\it proper}. See also \cite[formula (4.7)]{derdzinski-euh-kim-park}.

Proper weakly Ein\-stein metrics that are also locally homogeneous
were constructed by Euh {\it et al.}\
\cite[Expl.\,3.7]{euh-park-sekigawa-ms}, and classified by 
Arias-Marco and Kowalski \cite[p.\,23]{arias-marco-kowalski}. Other related
results can be found, for instance, in
\cite{caeiro-oliveira-marino-villar,marino-villar,wang-zhang}.

There exist proper weakly Ein\-stein K\"ah\-ler surfaces
\cite[Sect.\,12]{derdzinski-euh-kim-park}, and it is known
\cite[Theorem 1.2]{derdzinski-euh-kim-park} that they are never self-dual
relative to the standard orientation.

It remains an open question, however, whether a proper weakly Ein\-stein
K\"ah\-ler surface can be {\it anti-self-dual}. We conjecture that this is
{\it not\/} the case, which -- see Remark~\ref{equiv} below --
is nothing else than 
the following claim:
\begin{equation}\label{cnj}
\begin{array}{l}
\mathrm{any\hs\ scalar}\hyp\mathrm{flat\hs\ K}\ddot{\mathrm{a}}\mathrm{hler\ 
surface\hs\ in\ which\ the\ Weyl\ tensor\ acting}\\
\mathrm{on\ }\,2\hyp\mathrm{forms\ annihilates\ the\ Ric\-ci\
form\ must\ have\ parallel\ Ric\-ci}\\
\mathrm{tensor;\ this,\nh\ equivalently,\hn\ means\ that\ it\
is\ either\ Ric\-ci}\hyp\mathrm{flat,\nnh\ or\ }\\
\mathrm{locally\ isometric\ to\ a\ Riemannian\ product\ of\ two\ real\
surfaces}\\
\mathrm{with\ mutually\ opposite,\nnh\ nonzero\ constant\nh\ Gaussian\
curvatures.}
\end{array}
\end{equation}
The above phrasing of our conjecture shows that it is of independent
interest, aside from questions pertaining to weakly Ein\-stein manifolds.
With the symbols
\begin{equation}\label{smb}
M\nh,\,\,\,J,\,\,\,g,\,\,\nabla\nnh,\,\,R,\,\,\,\ric,\,\,W\nnh\nh,\,\,\,
\sca,\,\,\,\omega,\,\,\,\rho
\end{equation}
always standing for the underlying complex surface,
the com\-plex-struc\-ture tensor, the K\"ah\-ler metric in
question, its Le\-vi-Ci\-vi\-ta connection, curvature, Ric\-ci and Weyl
tensors, scalar curvature, K\"ah\-ler form and Ric\-ci form, 
Conjecture (\ref{cnj}) states -- according to formula (\ref{cff}) below --
that, in a K\"ah\-ler surface,
\begin{equation}\label{rds}
\mathrm{if\ }\,\sca\,\mathrm{\ and\ }\,W\nnh\nh\rho\,\mathrm{\ vanish\ 
identically,\ then\ so\ does\ }\,W\nnh\mathrm{\ or\ }\,\rho.
\end{equation}
Another version of (\ref{cnj}) uses the conditions
\begin{equation}\label{frm}
\mathrm{i)}\hskip5ptg(e\hn_j\w,e\hn_k\w)=\delta\hn_{jk}\w,\qquad
\mathrm{ii)}\hskip5pt(J\hn e\hn_1\w,\hs J\hn e_2\w,\hs J\hn e_3\w,\hs
J\hn e\nh_4\w)=(e_2\w,\hs-e\hn_1\w,\hs e\nh_4\w,\hs-e_3\w)
\end{equation}
imposed on a smooth local frame $\,e\hn_1\w,e_2\w,e_3\w,e\nh_4\w$, along
with the requirement that
\begin{equation}\label{rot}
\begin{array}{l}
\mathrm{the\,\ only\,\ nonzero\,\ curvature\,\ components}\\
R\hn_{ijkl}\w\,\mathrm{\ be\ the\ ones\ algebraically\
related\ to}\\
R\hn_{1212}\w=-\nnh\ly,\nnh\nh\quad 
R\hn_{1313}\w=R\hn_{2424}\w=R\hn_{1423}\w=\sy,\\
R\hn_{1414}\w=R\hn_{2323}\w=R\hn_{1342}\w=-\sy,\nnh\quad R\hn_{3434}\w=\ly
\end{array}
\end{equation}
for some functions $\,\ly,\sy$, where $\,R\hn_{ijkl}\w
=R(e\hn_i\w,e\nh_j\w,e\hn_k\w,e\hn_l\w)$. 
At the end of Sect.\,\ref{pr} we have this equivalent phrasing 
of Conjecture (\ref{cnj}): in an al\-most-com\-plex four-man\-i\-fold,
\begin{equation}\label{stt}
\mathrm{if\ \ (\ref{frm})\,}-\mathrm{\,(\ref{rot})\ \ hold\ and\ 
}\,\nabla\nnh J=0\mathrm{,\ then\ }\,\ly=0\,\mathrm{\ or\ }\,\sy=0\,\mathrm{\ 
identically.}
\end{equation}
Our argument for plausibility of Conjecture (\ref{cnj}) is twofold. First, 
we establish (\ref{cnj}) in the following
three special cases. For proofs, see Sect.\,\ref{pr}, \ref{pc} and \ref{ct}.
\begin{thm}\label{thrsc}
Conjecture\/ {\rm(\ref{cnj})} is true if one also assumes that
\begin{enumerate}
\item[{\rm(i)}] $M\,$ is compact, or
\item[{\rm(ii)}] one of the Ric\-ci eigen\-dis\-tri\-bu\-tions is
in\-te\-gra\-ble, or
\item[{\rm(iii)}] the norms of the Ric\-ci and Weyl tensors are functionally
dependent.
\end{enumerate}
The conclusion about\/ {\rm(i)} remains valid 
beyond the K\"ah\-ler
category. Namely, any compact oriented Riemannian four-man\-i\-fold which is 
sca\-lar-flat, anti-self-dual and weakly Ein\-stein, must be 
Ric\-ci-flat or con\-for\-mal\-ly flat.
\end{thm}
Both (ii), (iii) hold if the norm of the Ric\-ci tensor is constant
(Remark~\ref{lycst}).

Secondly, we provide evidence supporting Conjecture (\ref{cnj}) in the form
of calculations related to the 
system of exterior equations on a four-man\-i\-fold,
involving two functions $\,\ly,\sy\,$ and eight
$\,1$-forms $\,\aj,\bj,\cj,\dj,\fj\nh,\gj,\lj,\sj$, which reads
\begin{equation}\label{evp}
\begin{array}{l}
\mathrm{at\ each\ point\ }\,\aj,\bj,\cj,\dj\,\mathrm{\ are\ linearly\
independent,}\\
\mathrm{while,\ for\ }\,\ej\hn=(\sj+\lj)/2\,\mathrm{\ and\
}\,\hj\hn=(\sj-\lj)/2, \\
d\aj=\bj\wedge\ej+\cj\wedge\fj\nh+\dj\wedge\gj,\\
d\bj=-\hn\aj\wedge\ej+\dj\wedge\fj-\cj\wedge\gj,\\
d\hh\cj=-\hn\aj\wedge\fj\nh+\bj\wedge\gj+\dj\wedge\hj,\\
d\dj=-\bj\wedge\fj\nh-\aj\wedge\gj-\cj\wedge\hj,\\
d\fj\hs=\,\lj\wedge\gj\,
+\,\sy(\aj\wedge\cj\,-\,\dj\wedge\bj),\\
d\hh\gj\,=\,-\lj\wedge\fj\hs
\,-\sy(\aj\wedge\nh\dj\,-\,\bj\wedge\cj),\\
d\lj\,=\,-4\fj\nh\wedge\gj\,-\,\ly(\aj\wedge\bj+\cj\wedge\dj),\\
d\sj\,=\,-\ly(\aj\wedge\bj-\cj\wedge\dj).
\end{array}
\end{equation}
As we point out in Sect.\,\ref{ae}, Conjecture (\ref{cnj}) is equivalent
to the following statement: {\it the system\/} (\ref{evp}) {\it implies
that one of the functions\/ $\,\ly,\sy\,$ vanishes identically.} 

We emphasize here the {\it constructive aspect of the system\/} (\ref{evp}):
if Conjecture (\ref{cnj}) turns out to be false, (\ref{evp}), with
$\,\ly\sy\ne0$, will serve as
a local characterization, at generic points, of all the counterexamples,
that is, according to (\ref{rds}), of those K\"ah\-ler surfaces with 
$\,\sca=0\,$ and $\,W\nnh\nh\rho=0\,$ in which $\,W\nnh\nh\ne\hs0\,$ and
$\,\rho\ne0\,$ everywhere.

Geometrically, $\,\aj,\bj,\cj,\dj\,$ represent $\,J\,$ and $\,g\,$ by being 
$\,g$-or\-tho\-nor\-mal and having $\,(J^*\hskip-3pt\aj,J^*\nh\cj)
=(-\bj,-\hn\dj)$, while
$\,\ej\nh,\fj\nh,\gj,\hj\,$ are the connection $\,1$-forms in the frame
$\,e\hn_1\w,e_2\w,e_3\w,e\nh_4\w$ dual to $\,\aj,\bj,\cj,\dj$, there being
just four of them, as they together constitute a $\,1$-form valued in the Lie 
algebra $\,\mathfrak{u}(2)$. In terms of $\,e\hn_1\w,e_2\w,e_3\w,e\nh_4\w$, 
we have (\ref{frm}), and the first four exterior equations of
(\ref{evp}) state that $\,\nabla\hs$ is tor\-sion-free, while the last four 
amount precisely to (\ref{rot}). See Lemma~\ref{eqvto}.

Remark~\ref{omrho} provides another interpretation of (\ref{rot}): in
(\ref{smb}),
\begin{equation}\label{krf}
\omega=\aj\wedge\bj+\cj\wedge\dj,\qquad\rho=-\ly(\aj\wedge\bj-\cj\wedge\dj), 
\end{equation}
the length $\,\sqrt{2\,}$
mutually orthogonal anti-self-dual $\,2$-forms
\begin{equation}\label{egf}
\zeta\hs=\hs-\nnh\aj\wedge\bj+\cj\wedge\dj,\quad\eta\hs=\hs-\nnh\aj\wedge\cj\hn
+\hn\dj\wedge\bj,\quad\theta\hs=\hs-\nnh\aj\wedge\nh\dj\hs+\hn\bj\wedge\cj
\end{equation}
are eigenforms of $\,W\nnh\hn$ for the eigen\-value functions
$\,0,\,2\sy,-\nnh2\sy$, while $\,2\hh\gj,2\fj\nh,\lj\,$ are -- due to formula
(\ref{lmn}) -- the connection $\,1$-forms, relative to 
the local trivialization (\ref{egf}), of the Le\-vi-Ci\-vi\-ta connection in
the bundle of anti-self-dual $\,2$-forms. 

If $\,\ly\sy\ne0$, (\ref{evp}) has further consequences. Applying $\,d\,$ to
it, in Sect.\,\ref{ae}, we derive explicit rational expressions (\ref{sol})
for the component functions of $\,\lj,\fj\nh,\gj$ in terms of 
$\,\ly,\hs\sy,\,d\ly\,$ and $\,d\sy\hn$. These expressions turn 
(\ref{evp}) into a system of thir\-ty-six equations imposed on
$\,\ly,\sy$, the component functions $\,\sj\nh_i\w,\ly_i\w,\sy\nnh_i\w$
of $\,\sj,\,d\ly,\,d\sy\nh$, and their first-or\-der directional derivatives
$\,\sj\nh_{ij}\w,\ly\hh_{ij}\w,\sy\nnh_{ij}\w$. For now, let us choose to
ignore the six equations involving $\,\sj\nh_{ij}\w$. The remaing thirty
equations
imply that
\begin{enumerate}
\item[{\rm(a)}] $(12\sy^2\nh-\ly^2)(\ly_2\w\ly_3\w-\ly_1\w\ly_4\w)
=2\ly[\mu_-\w(\ly_2\w\sy\nnh_3\w-\ly_1\w\sy\hskip-1.5pt_4\w)
+\mu_+\w(\ly_3\w\sy\nnh_2\w-\ly_4\w\sy\nnh_1\w)]$,
\item[{\rm(b)}] $(12\sy^2\nh-\ly^2)(\ly_1\w\ly_3\w+\ly_2\w\ly_4\w)
=2\ly[\mu_+\w(\ly_1\w\sy\nnh_3\w+\ly_2\w\sy\hskip-1.5pt_4\w)
+\mu_-\w(\ly_3\w\sy\nnh_1\w+\ly_4\w\sy\nnh_2\w)]$,
\end{enumerate}
where $\,\mu_\pm\w=2\sy\pm\ly$. See Remark~\ref{cmbin}, where we also
describe the reason why the number of equations may be further reduced by 
six, to twen\-ty-four: the equations satisfy four 
lin\-e\-ar-de\-pen\-dence relations, and two of their functional combinations
result in the constraints (a) -- (b), thus allowing us to eliminate two more
equations, and replace them with (a) -- (b). However, taking the directional
derivatives of (a) and (b), we can in turn augment our
system with eight more equations, raising their total number to thir\-ty-two.
They are linear in the thir\-ty-six 
quantities 
$\,\sj\nh_i\w,\ly\hh_{ij}\w,\sy\nnh_{ij}\w$. This last number can be
reduced to thir\-ty-two since -- as explained in the lines following 
(\ref{lfz}) -- we are free to assume that $\,\ly_4\w=0\,$ by suppressing a
rotational symmetry inherent in the system (\ref{frm}) -- (\ref{rot}).
Consequently, 
there are now as many linear equations as there are unknowns. What makes it 
plausible to expect that no nontrivial solutions exist is the presence
of six more equations, involving $\,\sj\nh_{ij}\w$, which we earlier chose
to ignore.

In the case of {\it compact\/} four-man\-i\-folds, \cite[Prop.\,4.70]{besse},
reproduced below as formula (\ref{grd}), seems to hint at some variational
aspect of the weakly Ein\-stein condition. Closer inspection shows that 
this is not the case, in any interesting way: {\it no proper 
weakly Ein\-stein metric is a critical point of\/ $\,\|\hn R\|^2\nh$, the
squar\-ed $\,L\nnh^2$ norm of the curvature tensor}. See (\ref{npw}). On
the other hand, when the manifold is also oriented,
{\it sca\-lar-flat anti-self-dual Riemannian metrics, when they exist, are 
well known -- cf.\ {\rm(\ref{min})} -- to be precisely the absolute minima of
the functional\/} $\,\|\hn R\|^2\nh$. The last class includes sca\-lar-flat
K\"ah\-ler-sur\-face metrics, numerous nontrivial examples of which are
exhibited in \cite{lebrun} and \cite{kim-lebrun-pontecorvo}, while the two
italicized statements are used, in Sect.\,\ref{pr}, to prove 
the final clause of Theorem~\ref{thrsc}.

The authors wish to thank Claude LeBrun for helpful comments and
clarifications concerning the subject matter of the preceding paragraph
and the way in which the final clause of Theorem~\ref{thrsc} arises from
variational considerations.

\setcounter{thm}{0}
\renewcommand{\thethm}{\thesection.\arabic{thm}}
\section{Preliminaries}\label{pr}
\setcounter{equation}{0}
All manifolds are assumed connected, all mappings and tensor fields smooth.

We will use the fact that, given
a local frame $\,e\hn_i\w$ in a manifold and its dual $\,e^i\nh$,
\begin{equation}\label{cij}
2\hh d\hh e^k\nh=-C_{ij}^ke^i\wedge e^j\mathrm{\ for\ the\ functions\
}\,C_{ij}^k\mathrm{\ with\ }\,[\hh e\hn_i\w,e\nh_j\w]=C_{ij}^k e\hn_k\w.
\end{equation}
For the curvature, Weyl, and Ric\-ci tensors of a Riemannian metric $\,g\,$
in any dimension $\,n\ge3\,$ we use the sign convention such 
that $\,\ric\hn_{ij}\w=g^{pq}\nh R_{ipjq}\w$. Thus, 
\[
\begin{array}{l}
W_{\!ijpq}\w=R_{ijpq}\w
-\displaystyle{\frac1{n-2}}\,(g_{ip}\w\ric_{jq}\w
+g_{jq}\w\ric_{ip}\w-g_{jp}\w\ric_{iq}\w
-g_{iq}\w\ric_{jp}\w)\\
\hskip61.5pt
+\,\displaystyle{\frac{\sca}{(n-1)(n-2)}}
\hs(g_{ip}\w g_{jq}\w-g_{jp}\w g_{iq}\w)\hh,
\end{array}
\]
$\sca\,$ being the scalar curvature. Consequently,
if $\,n=4\,$ and $\,e\hn_1\w,e_2\w,e_3\w,e\nh_4\w$ is a local 
$\,g$-or\-tho\-nor\-mal frame, one has (\ref{rot}) if and only if 
\begin{equation}\label{nzc}
\begin{array}{l}
\mathrm{the\ only\ nonzero\ components\ of\ the\ Weyl\ and}\\
\mathrm{Ric\-ci\ tensors\ are\ those\ algebraically\
related\ to}\\ 
W_{\!1313}\w=W_{\!2424}\w=W_{\!1423}\w=\sy,\,\,\quad\ric_{11}\w=\ric_{22}\w
=-\nnh\ly,\\
W_{\!1414}\w=W_{\!2323}\w=W_{\!1342}\w=-\sy,\,\,\quad\ric_{33}\w
=\ric\hn_{44}\w=\ly.
\end{array}
\end{equation}
As shown by Tanno \cite{tanno}, a  
con\-for\-mal\-ly flat K\"ah\-ler surface is
\begin{equation}\label{cff}
\begin{array}{l}
\mathrm{either\hs\ flat,\,\hh\ or\hs\ locally\hs\ isometric\hs\ to\hs\ a\hs\
Riemannian\hs\ product\hs\ of\hs\ two\hs\ real}\\
\mathrm{surfaces\nh\ with\nh\ mutually\nh\ opposite,\nh\nnh\ nonzero\nh\
constant\nh\ Gaussian\nh\ curvatures.}
\end{array}
\end{equation}
On the other hand, it is well known -- see, for instance,
\cite[p.\,459]{derdzinski-00} -- that
\begin{equation}\label{asd}
\mathrm{anti}\hyp\mathrm{self}\hyp\mathrm{dual\-i\-ty\ of\ a\
K}\ddot{\mathrm{a}}\mathrm{hler\ surface\ amounts\ to\ its\
sca\-lar}\hyp\mathrm{flat\-ness.}
\end{equation}
\begin{rem}\label{equiv}According to parts (a) and (d) in 
\cite[Theorem 5.1]{derdzinski-euh-kim-park}, a sca\-lar-flat K\"ah\-ler 
surface is weakly Ein\-stein if and only if $\,W\nnh\nh\rho=0$.
Thus, (\ref{cnj}) is, due to (\ref{asd}) and (\ref{cff}), 
equivalent to the conjecture stated immediately before (\ref{cnj}).
\end{rem}
\begin{proof}[Proof of Theorem~\ref{thrsc}, {\rm case (i)}]The equivalence
between
parts (d) and (e) in \cite[Theorem 5.1]{derdzinski-euh-kim-park} shows that a
K\"ah\-ler surface with $\,\sca=0\,$ and $\,W\nnh\nh\rho=0$ necessarily has
$\,\Delta\ry=0\,$ or, in local coordinates, $\,\ry_{ij,k}\w{}^k\nh=0$.
In the compact case, integrating the inner product
$\,\langle\ry,\Delta\ry\rangle$ by parts, we get $\,\nabla\nnh\ry=0$, as
required.
\end{proof}
In the following argument $\,\mathrm{d}^*$ is the formal ad\-joint of the 
{\it Co\-daz\-zi operator\/} $\,\mathrm{d}$, cf.\
\cite[Sect.\,16.5]{besse}, sending a 
twice-co\-var\-i\-ant symmetric tensor field $\,\sym\,$ on a Riemannian 
manifold to $\,\mathrm{d}\sym\,$ given, in coordinates, by 
$\,[\mathrm{d}\sym]_{ijk}\w=\sym_{kj,\hh i}\w-\,\sym_{ki,\hs j}\w$.
\begin{proof}[Proof\hn\ of\hn\ the\hn\ final\hn\ clause\hn\ of\hn\
Theorem~\ref{thrsc}]From the Chern-Weil integral formulae for characteristic
numbers one obtains
\[
12\hh\|\hn R\|^2\nh 
=\|\sca\|^2\nh+48\hh\|W^+\nnh\|^2\nh-96\pi^2[\hh\chi(M)+3\vt(M)],
\]
where $\,\chi(M)\,$
and $\,\vt(M)\,$ are the Eu\-ler characteristic and signature of a compact
oriented Riemannian four-man\-i\-fold $\,(M\nh,g)$, and $\,\|\hskip2.3pt\|\,$
denotes the $\,L\nnh^2$ norm. One verifies this, for instance, by subtracting 
12 times \cite[formula\,(25.1)]{derdzinski-00} from
\cite[formula\,(25.8)]{derdzinski-00}. Thus, given a compact oriented
four-man\-i\-fold $\,M\nh$,
\begin{equation}\label{min}
\begin{array}{l}
\mathrm{if\hh\ }\,M\,\mathrm{\ admits\hh\ a\hh\ scalar}\hyp\mathrm{flat\
anti}\hyp\mathrm{self}\hyp\mathrm{dual\hh\ Riemannian\hh\ me}\hyp\\
\mathrm{tric\ }\,\,g\mathrm{,\ any\ such\ }\hs\,g\hs\,\mathrm{\ realizes\ the\ 
absolute\ minimum\ value}\\
\mathrm{of\hh\ the\hh\ functional\hh\ }\hs\|\hn R\|^2\nh\mathrm{\ in\hh\ the\hh\
space\hh\ of\hh\ all\hh\ metrics\hh\ on\nnh\ }\,M\nh,
\end{array}
\end{equation}
the minimum being $\,-8\pi^2[\hh\chi(M)+3\vt(M)]$. On the other hand, 
according to \cite[Prop.\,4.70]{besse}, on a compact manifold of dimension
four, with $\,\mathrm{d},\mathrm{d}^*$ defined above,
\begin{equation}\label{grd}
\begin{array}{l}
\mathrm{at\ any\ metric\ }\,g\mathrm{,\ one}\hyp\mathrm{half\ of\ the\
gradient\ of\ the}\\
\mathrm{functional\hn\ }\,\|\hn R\|^2\mathrm{\hn\ equals\hn\ 
}\,\mathrm{d}^*\hn\mathrm{d}\hh\ric\,\mathrm{\hn\ minus\hn\ the\hn\
trace\-less}\\
\mathrm{part\ of\ the\ triple\ contraction\ of\ }\nh\,R\nh\,\mathrm{\ against\
itself.}
\end{array}
\end{equation}
Thus, if $\,g\,$ satisfies the assumptions in the final clause of
Theorem~\ref{thrsc}, the resulting equality 
$\,\mathrm{d}^*\hn \mathrm{d}\hh\ric=0\,$ implies, via integration by parts,
that $\,\mathrm{d}\hh\ric=0$, and so $\,g\,$ has harmonic curvature
\cite[Sect.\,16.33]{besse}.
Our assertion now follows from \cite[Thm.\,22.3]{derdzinski-park-shin-cp}.
In the K\"ah\-ler case, instead of invoking \cite{derdzinski-park-shin-cp},
one can also use \cite[Prop.\,16.30]{besse}.
\end{proof}
The third and second lines above show that, on a compact 
four-man\-i\-fold,
\begin{equation}\label{npw}
\mathrm{no\nh\ proper\nh\ weakly\nh\ Ein\-stein\nh\ metric\nh\ is\nh\
 a\nh\ critical\nh\ point\ of\nh\ the\nh\ functional\nh\ }\,\|\hn R\|^2\nnh.
\end{equation}
\begin{rem}\label{lycst}The line following Theorem~\ref{thrsc} is obvious:
by (\ref{nzc}), the Ric\-ci tensor $\,\ric\,$ has the spectrum 
$\,-\nnh\ly,-\nnh\ly,\ly,\ly\,$ and the norm $\,2|\hn\ly|$, the constancy
of which trivially implies (iii), as well as in\-te\-gra\-bi\-li\-ty\ of 
both Ric\-ci eigen\-dis\-tri\-bu\-tions when $\,\ly\ne0$, since they
are the kernels of the closed $\,2$-forms $\,\rho\pm\hn\nnh\ly\hh\omega$, cf.\
(\ref{smb}).
\end{rem}
\begin{lem}\label{rijij}At points of an anti-self-dual
K\"ah\-ler surface\/ $\,(M\nh,g)\,$
where the Ric\-ci and Weyl tensors are both nonzero, the weakly Ein\-stein
property is equivalent, locally, to the existence of a smooth frame\/ 
$\,e\hn_1\w,e_2\w,e_3\w,e\nh_4\w$ such that\/ {\rm(\ref{frm})}
-- {\rm(\ref{rot})} hold with some
no\-where-zero functions\/ $\,\ly\,$ and\/ $\,\sy$.
\end{lem}
\begin{proof}We assume that $\,(M\nh,g)\,$ is anti-self-dual, and so
$\,\sca=0\,$ due to (\ref{asd}).

First, let $\,(M\nh,g)\,$ be weakly Ein\-stein
with $\,W\nnh\hn$ and $\,\ric\,$ both nonzero at $\,x\in M\nh$. As
$\,\sca=0$, the final clause of
\cite[Corollary 1.6]{derdzinski-park-shin-ct} allows us to choose
scalars $\,\ly,\sy\,$ and a positive
orthonormal basis $\,e\hn_1\w,e_2\w,e_3\w,e\nh_4\w$ of $\,\txm\,$ with
(\ref{nzc}), where \cite{derdzinski-park-shin-ct} uses the symbol $\,\xi\,$
for our $\,\sy$. Thus, (\ref{rot}) follows. 
Switching $\,e\hn_1\w,e_2\w$ with $\,e_2\w,e\hn_1\w$ 
and $\,e_3\w,e\nh_4\w$ with $\,e\nh_4\w,e_3\w$, if needed,
we also get (\ref{frm}-ii).
Since $\,\ly\sy\ne0$, we may choose such $\,e\hn_1\w,e_2\w,e_3\w,e\nh_4\w$
which depend smoothly on the point $\,x$, as the resulting local frames  
are well known \cite[Sect.\,6]{derdzinski-piccione-terek} to be precisely the
smooth local sections of a $\,G$-prin\-ci\-pal bundle over $\,M\nh$, for some 
matrix group $\,G\subseteq\mathrm{SO}\hh(4)$.

Conversely, as (\ref{frm}) -- (\ref{rot}) lead to (\ref{nzc}), the first part
of \cite[Corollary 1.6]{derdzinski-park-shin-ct} for $\,\xi=\sy$, combined
with its final clause, implies the weakly Ein\-stein property of $\,g$. This 
completes the proof.
\end{proof}
We can now show that (\ref{stt}) is equivalent to Conjecture (\ref{cnj})
or, in other words (see Remark~\ref{equiv}) to the claim 
immediately preceding (\ref{cnj}). Namely, {\it the negation of\/}
(\ref{stt}) yields a K\"ah\-ler-sur\-face metric $\,g\,$ realizing
(\ref{frm}) -- (\ref{rot}), with $\,\ly\sy\ne0$, so that, according to
Lemma~\ref{rijij}, $\,g\,$ is weakly Ein\-stein and, by (\ref{nzc}), also
proper. On the other hand, any $\,g\,$ having this latter property, satisfies,
again due to Lemma~\ref{rijij} and (\ref{nzc}), the negation of
(\ref{stt}).

\section{Pointwise symmetries of the system\/ {\rm(\ref{frm})} --
{\rm(\ref{rot})}}\label{ps}
\setcounter{equation}{0}
\begin{rem}\label{rewrt}For a smooth local frame
$\,e\hn_1\w,e_2\w,e_3\w,e\nh_4\w$ in a K\"ah\-ler surface and functions
$\,\ly,\sy\nh$, the conditions (\ref{frm}) -- (\ref{rot}) are equivalent to 
having (\ref{frm}) along with the following version of (\ref{rot}):
\begin{equation}\label{rwf}
\begin{array}{l}
R\hn_{1313}\w\nh=\nh R\hn_{1441}\w\nh=\,\sy,\quad
R\hn_{1221}\w=R\hn_{3434}\w=\hs\ly,\quad R\hn_{1234}\w\nh=\,0,\\
R_{ijkl}\w\nh=0\,\mathrm{\ \ whenever\ the\ set\  
}\,\{i,j,k,l\}\,\mathrm{\ has\ }\hs3\hs\mathrm{\ elements,}
\end{array}
\end{equation}
as the remaining parts of (\ref{rot}) then are immediate from the
general Riemannian and K\"ahler symmetries of the curvature tensor.
\end{rem}
\begin{lem}\label{symtr}The conditions\/ {\rm(\ref{frm})} -- {\rm(\ref{rot})}
imposed on\/ $\,g,J,e\hn_1\w,e_2\w,e_3\w,e\nh_4\w,\ly\,$ and\/ $\,\sy\,$
remain satisfied when the sextuple\/
$\,(e\hn_1\w,e_2\w,e_3\w,e\nh_4\w,\ly,\sy)\,$ is replaced by one of
\begin{equation}\label{rep}
\begin{array}{rlrl}
\mathrm{(i)}&(e_2\w,-e\hn_1\w,e_3\w,e\nh_4\w,\ly,-\nh\sy),\qquad&
\mathrm{(ii)}&(e\hn_1\w,e_2\w,e\nh_4\w,-e_3\w,\ly,-\nh\sy),\\
\mathrm{(iii)}&(e_2\w,-e\hn_1\w,e\nh_4\w,-e_3\w,\ly,\sy),\qquad&
\mathrm{(iv)}&(e_3\w,e\nh_4\w,e\hn_1\w,e_2\w,-\nnh\ly,\sy).
\end{array}
\end{equation}
\end{lem}
In fact, with (\ref{rot}) replaced by (\ref{rwf}): the claims about (i) and
(iv) are obvious,
\begin{equation}\label{cng}
\begin{array}{l}
\mathrm{(ii)\ equals\ the\ result\ of\ conjugating\ (i)\ by\
(iv),}\\
\mathrm{and\hn\ (iii)\hn\ is\hn\ the\hn\ composition\hskip-3pt:\hskip-2.2pt\ 
(i)\hn\ followed\hn\ by\hn\ (ii).}
\end{array}
\end{equation}
More generally, (i) and (iv) are easily seen to
generate a thir\-ty-two-el\-e\-ment group
acting on the sextuples $\,(e\hn_1\w,e_2\w,e_3\w,e\nh_4\w,\ly,\sy)$.
In addition, (iii) is a special case of
a {\it continuous symmetry\/} of the system 
(\ref{frm}) -- (\ref{rot}):
\begin{lem}\label{cnsym}With\/ $\,\ly\,$ and $\,\sy\,$ fixed, 
{\rm(\ref{frm})} -- {\rm(\ref{rot})} are invariant under
simultaneous rotations, involving any angle function, of the pairs\/ 
$\,(e\hn_1\w,e_2\w)\,$ and\/ $\,(e_3\w,e\nh_4\w)$.
\end{lem}
\begin{proof}Replacing (\ref{rot}) with (\ref{rwf}), we clearly get our claim
for the equalities in (\ref{rwf}) having $\,\ly\,$ or $\,0\,$ on the
right-hand side. For the rotational invariance of
$\,R\hn_{1313}\w$ and $\,R\hn_{1414}\w$, it suffices to
verify that, whenever $\,c^2\nh+s^2\nh=1$,
\begin{equation}\label{rce}
R(c\hs e\hn_1\w+s\hh e_2\w,\,c\hs e_3\w+s\hh e\nh_4\w,\,
c\hs e\hn_1\w+s\hh e_2\w,\,c\hs e_3\w+s\hh e\nh_4\w)=R\hn_{1313}\w,
\end{equation}
as the analogous equality for $\,R\hn_{1414}\w$ arises 
when one replaces $\,(e_3\w,e\nh_4\w)\,$ by $\,(e\nh_4\w,-e_3\w)$
and invokes (\ref{rep}-ii). The left-hand side of (\ref{rce}), easily
evaluated from (\ref{rwf}), equals $\,(c^2\nh+s^2)^2\nh\sy\nh$,
completing the proof.
\end{proof}
\begin{rem}\label{simrt}At any given point, the conditions
(\ref{frm}) -- (\ref{rot}) with $\,\ly\sy\ne0$ uniquely determine 
the frame 
$\,e\hn_1\w,e_2\w,e_3\w,e\nh_4\w$
up to simultaneous rotations of the pairs 
$\,(e\hn_1,e_2\w)\,$ and $\,(e_3\w,e\nh_4\w)$, as in Lemma~\ref{cnsym}, 
and the action of the thir\-ty-two-el\-e\-ment group, mentioned
in the line following (\ref{cng}). See the proof of (c) in 
\cite[the lines following formula (7.7)]{derdzinski-park-shin-ct},
applied to $\,\theta=0$, which could also be used as another proof of our 
Lemma~\ref{cnsym}.
\end{rem}

\section{The associated exterior differential equations}\label{ae}
\setcounter{equation}{0}
In a Riemannian four-man\-i\-fold $\,(M\nh,g)\,$
carrying an al\-most com\-plex structure
$\,J$ compatible with the metric $\,g$, consider a smooth local frame
$\,e\hn_1\w,e_2\w,e_3\w,e\nh_4\w$ satisfying (\ref{frm}), and a connection 
$\,\nabla\hs$ in the tangent bundle $\,T\nh M\,$ with $\,\nabla\nh g=0\,$ and
$\,\nabla\nnh J=0$, having the component
functions $\,\vg_{\hskip-2.2ptij}^k$ and connection $\,1$-forms
$\,\vg_{\hskip-2.2ptj}^k$ given by
\begin{equation}\label{cfo}
\nabla\nh e\nh_j\w\nh=\vg_{\hskip-2.2ptj}^k\nnh\nh\otimes e\hn_k\w\mathrm{\
(summed\ over\ }\,k\mathrm{),\ so\ that\ }\vg_{\hskip-2.2ptj}^k(e\hn_i\w)
=\vg_{\hskip-2.2ptij}^k\mathrm{\ and\
}\hs\nabla\hskip-2.7pt_{e\hn_i\w}\w\nnh\nh e\nh_j\w\hn
=\hn\vg_{\hskip-2.2ptij}^ke\hn_k\w.
\end{equation}
Then $\,\vg_{\hskip-2.2ptj}^k$ 
may be arranged in a matrix of the form
\begin{equation}\label{mtx}
\left[\begin{matrix}
\vg_{\hskip-2.2pt1}^1&\vg_{\hskip-2.2pt2}^1
&\vg_{\hskip-2.2pt3}^1&\vg_{\hskip-2.2pt4}^1\cr
\vg_{\hskip-2.2pt1}^2&\vg_{\hskip-2.2pt2}^2
&\vg_{\hskip-2.2pt3}^2&\vg_{\hskip-2.2pt4}^2\cr
\vg_{\hskip-2.2pt1}^3&\vg_{\hskip-2.2pt2}^3
&\vg_{\hskip-2.2pt3}^3&\vg_{\hskip-2.2pt4}^3\cr
\vg_{\hskip-2.2pt1}^4&\vg_{\hskip-2.2pt2}^4
&\vg_{\hskip-2.2pt3}^4&\vg_{\hskip-2.2pt4}^4\end{matrix}\right]\nnh\nh
=\nnh\nh
\left[\begin{matrix}
0&\ej&\fj&\gj\cr
-\ej&0&-\gj&\fj\cr
-\fj&\gj&0&\hj\cr
-\gj&-\fj&-\hj&0\end{matrix}\right]\nnh\nnh.
\end{equation}
In fact, for arbitrary
functions $\,\vg_{\hskip-2.2ptij}^k$ in (\ref{cfo}), it is clear that
$\,\nabla\nh g=0\,$ if and only if 
$\,\vg_{\hskip-2.2ptij}^k$ is skew-sym\-met\-ric in $\,j,k$, while 
$\,\nabla\nnh J=0\,$ if and only if
$\,\vg_{\hskip-2.2ptij}^k=(-\nnh1)^{j+k}\nh\vg_{\hskip-2.2ptip}^q$
whenever $\,\{\{j,p\},\nh\{k,q\}\}\nh=\hn\{\{1,2\},\nh\{3,4\}\}$.
Note that $\,\nabla\hs$ 
need not be tor\-sion-free.

In other words, $\,\nabla\hs$ is a connection 
in $\,T\nh M\,$ viewed as a complex vector bundle, in which 
$\,g\,$ is the real part of a $\,\nabla\nh$-par\-al\-lel 
Her\-mit\-i\-an fibre metric, and the right-hand side of (\ref{mtx})
represents, in real terms, a $\,1$-form valued in the Lie 
algebra $\,\mathfrak{u}(2)$.

Let 
$\,\aj,\bj,\cj,\dj$ be the local trivialization of $\,\tam\,$ 
dual to the frame $\,e\hn_1\w,e_2\w,e_3\w,e\nh_4\w$:
\begin{equation}\label{tfr}
(\aj,\bj,\cj,\dj)=(e^1\nh,e^2\nh,e^3\nh,e^4).
\end{equation}
\begin{rem}\label{omrho}With (\ref{tfr}), the relations (\ref{frm}) and 
(\ref{nzc}) easily give 
$\,g=\aj\otimes\hn\aj+\bj\otimes\hn\bj+\cj\hn\otimes\cj+\dj\otimes\nh\dj\,$
and $\,\ric=-\nnh\ly(\aj\otimes\hn\aj+\bj\otimes\hn\bj)
+\nnh\ly(\cj\hn\otimes\cj+\dj\otimes\nh\dj)$
in (\ref{smb}), while
$\,(W\nnh\nh\zeta,W\nnh\hn\eta,W\nnh\hn\theta)
=(0,\,2\sy\hn\eta,-\nnh2\sy\theta)$,
proving (\ref{krf}) and the claim following (\ref{egf}). 
\end{rem}
\begin{lem}\label{smtrs}Under the above assumptions, with\/ 
$\,\sj=\ej\hn+\nh\hj\nh$, the 
replacements\/ {\rm(i)} -- {\rm(iv)} in Lemma\/~{\rm\ref{symtr}} cause\/
$\,e\hn_1\w,e_2\w,e_3\w,e\nh_4\w,\ly,\sy,\ej\nh,\hj\nh,\sj\,$ to be
replaced by
\begin{equation}\label{rpl}
\begin{array}{rlllll}
\mathrm{i)}&e_2\w,-e\hn_1\w,e_3\w,e\nh_4\w,
&\mathrm{and}&\ly,-\sy,
&\mathrm{and}&\ej\nh,\hj,\sj,\\
\mathrm{ii)}&e\hn_1\w,e_2\w,e\nh_4\w,-e_3\w,
&\mathrm{and}&\ly,-\sy,
&\mathrm{and}&\ej\nh,\hj,\sj,\\
\mathrm{iii)}&e_2\w,-e\hn_1\w,e\nh_4\w,-e_3\w,
&\mathrm{and}&\ly,\sy,
&\mathrm{and}&\ej\nh,\hj,\sj,\\
\mathrm{iv)}&e_3\w,e\nh_4\w,e\hn_1\w,e_2\w,
&\mathrm{and}&-\nnh\ly,\sy,
&\mathrm{and}&\hj,\ej\nh,\sj,\\
\mathrm{v)}&e_3\w,e\nh_4\w,e_2\w,-e\hn_1\w,
&\mathrm{and}&-\nnh\ly,-\sy,
&\mathrm{and}&\hj,\ej\nh,\sj.
\end{array}
\end{equation}
\end{lem}
\begin{proof}The cases of (i) and (iv) are obvious from
Lemma~\ref{symtr} and (\ref{mtx}), while (ii) -- (iii) then arise via
(\ref{cng}), and (v) is nothing else than (i) followed by (iv).
\end{proof}
\begin{rem}\label{sweep}Also, 
$\,\aj,\bj,\cj,\dj,\ej\nh,\fj\nh,\gj,\hj,\lj=\ej\hn-\hj\nh,\sj,\ly,\sy\,$ 
may be replaced 
\begin{equation}\label{swp}
\begin{array}{rlllll}
\mathrm{i)}&\bj,-\nnh\aj,\cj,\dj,
&\mathrm{and}&\ej\nh,-\gj,\fj\nh,\hj,\lj,\sj,
&\mathrm{and}&\ly,-\sy,\\
\mathrm{ii)}&\aj,\bj,\dj,-\cj,
&\mathrm{and}&\ej\nh,\gj,-\fj\nh,\hj,\lj,\sj,
&\mathrm{and}&\ly,-\sy,\\
\mathrm{iii)}&\dj,-\cj,\aj,\bj,
&\mathrm{and}&\hj,-\gj,-\fj\nh,\ej\nh,-\lj,\sj,
&\mathrm{and}&-\nnh\ly,-\sy,\\
\mathrm{iv)}&\cj,\dj,\aj,\bj,
&\mathrm{and}&\hj,-\fj\nh,\gj,\ej\nh,-\lj,\sj,
&\mathrm{and}&-\nnh\ly,\sy.
\end{array}
\end{equation}
\end{rem}
In fact, (iii) arises as (ii) followed by (iv), while the other claims
are immediate from Lemma~\ref{smtrs} and (\ref{mtx}) -- (\ref{tfr}).
\begin{lem}\label{eqvto}
Let\/
$\,M\nh,g,J,\nabla\nnh,\aj,\bj,\cj,\dj,\ej\nh,\fj\nh,\gj,\hj,e\hn_1\w,e_2\w,
e_3\w,e\nh_4\w$ be as above.
\begin{enumerate}
\item[{\rm(a)}] $\nabla\hs$ is tor\-sion-free if and only if
$\,\aj,\bj,\cj,\dj,\ej\nh,\fj\nh,\gj,\hj\,$
satisfy the first four exterior equations of\/ 
{\rm(\ref{evp})}.
\item[{\rm(b)}] If\/ $\,\nabla\hs$ is tor\-sion-free, {\rm(\ref{rot})} 
is equivalent to the last four equations in\/ {\rm(\ref{evp})} with\/
$\,\lj=\ej\hn-\hj\,$ and\/ $\,\sj=\ej\hn+\nh\hj$.
\end{enumerate}
\end{lem}
\begin{proof}By (\ref{cij}), $\,\nabla\hs$ is tor\-sion-free if and only if
$\,2\hh d\hh e^k\nh
=[\vg_{\hskip-2.2ptji}^k\nh-\vg_{\hskip-2.2ptij}^k]\hh e^i\nnh\wedge e^j$ 
which, due to skew-sym\-me\-try of $\,e^i\nnh\wedge e^j$ in $\,i,j$, is the
same as $\,d\hh e^k\nh=\vg_{\hskip-2.2ptji}^ke^i\nnh\wedge e^j\nh$, that is,
$\,d\hh e^k\nh=e^i\nnh\wedge\vg_{\hskip-2.2pti}^k\nnh$, since (\ref{cfo})
gives $\,\vg_{\hskip-2.2ptji}^ke^j\nh=\vg_{\hskip-2.2pti}^k\nnh$. Now 
(\ref{mtx}) and (\ref{tfr}) yield (a).

Next, let $\,\nabla\hs$ be tor\-sion-free. 
The components $\,R\hn_{ijk}\w{}^{\hs l}$ of the $\,(1,3)\,$
curvature tensor $\,R\,$ of $\,\nabla\nh$, with 
$\,R(e\hn_i\w,e\nh_j\w)\hh e\hn_k=R\hn_{ijk}\w{}^{\hs l}e\hn_l\w$, are given by
$\,R\hn_{ijk}\w{}^{\hs l}=d\nh_j\w\vg_{\hskip-2.2ptik}^{\hs l}
-d_i\w\vg_{\hskip-2.2ptjk}^{\hs l}
+\vg_{\hskip-2.2ptjp}^{\hs l}\vg_{\hskip-2.2ptik}^p
-\vg_{\hskip-2.2ptip}^{\hs l}\vg_{\hskip-2.2ptjk}^p
+[\vg_{\hskip-2.2ptij}^p-\vg_{\hskip-2.2ptji}^p]
\vg_{\hskip-2.2ptpk}^{\hs l}$, where $\,d_i\w$ denotes the
$\,e\hn_i\w$-di\-rec\-tion\-al derivative. This amounts to 
$\,R_{\nnh k}^{\hs l}=-d\vg_{\hskip-2.2ptk}^{\hs l}
+\vg_{\hskip-2.2ptk}^p\wedge\vg_{\hskip-2.2ptp}^{\hs l}$ for the $\,2$-forms 
$\,R_{\nnh k}^{\hs l}$, skew-sym\-met\-ric in $\,k,l$, characterized by 
$\,R_{\nnh k}^{\hs l}(e\hn_i\w,e\nh_j\w)=R\hn_{ijk}\w{}^{\hs l}$. As
(\ref{rot}) reads
\[
\begin{array}{l}
R_1^2=-\ly\aj\wedge\bj,\quad 
R_1^{\hs3}=R_2^{\hs4}=\sy(\aj\wedge\cj-\dj\wedge\bj),\\
R_3^{\hs4}=\ly\hh\cj\wedge\dj,\quad
R_2^{\hs3}=-R_1^{\hs4}=\sy(\aj\wedge\nh\dj-\bj\wedge\cj),
\end{array}
\]
the assertion (b) follows from (\ref{mtx}).
\end{proof}
\begin{lem}\label{lfgeq}Given functions\/ $\,\ly,\sy\,$ and 
$\,1$-forms\/ $\,\aj,\bj,\cj,\dj,\fj\nh,\gj,\lj,\sj,$ on a four-man\-i\-fold
satisfying\/ {\rm(\ref{evp})}, let us 
express\/ $\,\lj,\fj,\gj,\hh d\ly,\hh d\sy\,$ as functional combinations 
of\/ {\rm(\ref{tfr})}, 
with some coefficients\/  
$\,\lj_i\w,\fj\hskip-2.3pt_i\w,\gj\hn_i\w,\ly_i\w,\sy\nnh_i\w$, $\,i=1,2,3,4$. 
Then
\begin{equation}\label{sol}
\begin{array}{rl}
\mathrm{a)}&16\ly\sy^2\hn(\lj_1\w,\hs\lj_2\w,\hs\lj_3\w,\hs\lj_4\w)
=8\ly\sy(-\sy\nnh_2\w,\,\sy\nnh_1\w,\,\sy\hskip-1.5pt_4\w,\,-\sy\nnh_3\w)\\
&\phantom{16\ly\sy^2\hn(\lj_1\w,\hs\lj_2\w,\hs\lj_3\w,\hs\lj_4\w)}
\hs+\mu_*\w(\ly_2\w,\,-\nnn\ly_1\w,\,-\nnn\ly_4\w,\,\ly_3\w),\\
\mathrm{b)}&8\ly\sy(\fe,\hs\fz,\hs\gj\nh_3\w,
\hs\gj\nnh_4\w)=-\mu_-\w(\ly_3\w,\,\ly_4\w,\,\ly_2\w,\,-\nnn\ly_1\w),\\
\mathrm{c)}&8\ly\sy(\fd,\hs\fv,\hs\gj\nh_1\w,
\hs\gj\nh_2\w)=\mu_+\w(\ly_1\w,\,\ly_2\w,\hs-\nnn\ly_4\w,\,\ly_3\w),
\end{array}
\end{equation}
where\/ $\,\mu_\pm\w=2\sy\pm\ly\,$ and\/
$\,\mu_*\w=\mu_+\w\mu_-\w=4\sy^2\nh-\ly^2\nh$.
\end{lem}
\begin{proof}For the coefficient functions 
$\,\lj_i\w,\fj\hskip-2.3pt_i\w,\gj\hn_i\w,\ly_i\w,\sy\nnh_i\w$, one has
\begin{equation}\label{nel}
\begin{array}{rlrl}
\mathrm{i)}&2\sy\nh\lj_1\w+\mu_+\w\gj\nh_3\w=-\sy\nnh_2\w,&\mathrm{ii)}&
2\sy\nh\lj_2\w+\mu_+\w\gj\nnh_4\w=\sy\nnh_1\w,\\
\mathrm{iii)}&2\sy\nh\lj_3\w-\mu_-\w\gj\nh_1\w=\sy\hskip-1.5pt_4\w,&\mathrm{iv)}&
2\sy\nh\lj_4\w-\mu_-\w\gj\nh_2\w=-\sy\nnh_3\w,\\
\mathrm{v)}&\mu_+\w\fe+\mu_-\w\gj\nh_2\w=0,&\mathrm{vi)}&
\mu_+\w\fz-\mu_-\w\gj\nh_1\w=0,\\
\mathrm{vii)}&\mu_+\w\fd-\mu_-\w\gj\nnh_4\w=\ly_1\w,&\mathrm{viii)}&
\mu_+\w\fv+\mu_-\w\gj\nh_3\w=\ly_2\w,\\
\mathrm{ix)}&\mu_-\w\fe+\mu_+\w\gj\nh_2\w=\ly_3\w,&\mathrm{x)}&
\mu_-\w\fz-\mu_+\w\gj\nh_1\w=\ly_4\w,\\
\mathrm{xi)}&\mu_-\w\fd-\mu_+\w\gj\nnh_4\w=0,&\mathrm{xii)}&
\mu_-\w\fv+\mu_+\w\gj\nh_3\w=0.
\end{array}
\end{equation}
In fact, applying $\,d\,$ to the formula for $\,d\fj\hs$ in (\ref{evp}), 
as well as to 
$\,d(\sj-\lj)=4\fj\nh\wedge\gj+2\ly\hh\cj\wedge\dj\,$ and 
$\,d(\sj+\lj)=-4\fj\nh\wedge\gj-2\ly\aj\wedge\bj$, 
the dif\-fer\-ence and sum of the last two equations in 
(\ref{evp}), we obtain, respectively, from (\ref{evp}),
\[
\begin{array}{l}
d\sy\hn\wedge(\aj\wedge\cj\nh-\nh\dj\wedge\bj)\hn
=\hn\mu_+\w\gj\wedge\aj\wedge\bj\nh-\nh\mu_-\w\gj\wedge\cj\wedge\hn\dj
\nh-\nh2\sy\lj\wedge(\aj\wedge\nh\dj\nh-\nh\bj\wedge\cj),\\
d\ly\wedge\cj\wedge\hn\dj=-\mu_-\w\gj\wedge(\aj\wedge\cj-\dj\wedge\bj)
-\mu_+\w\fj\nh\wedge(\aj\wedge\nh\dj-\bj\wedge\cj),\\
d\ly\wedge\aj\wedge\bj\hs=\hs-\mu_-\w\fj\nh\wedge(\aj\wedge\nh\dj-\bj\wedge\cj)
-\mu_+\w\gj\wedge(\aj\wedge\cj-\dj\wedge\bj).
\end{array}
\]
These equalities amount to (\ref{nel}). 
As equations (\ref{nel}-v) -- (\ref{nel}-xii) read
\[
\begin{array}{l}
\left[\begin{matrix}
\mu_+\w&\mu_-\w\cr
\mu_-\w&\mu_+\w\end{matrix}\right]\nnh\nnh
\left[\begin{matrix}
\fd&\fv&\gj\nh_2\w&-\gj\nh_1\w\cr
-\gj\nnh_4\w&\gj\nh_3\w&\fe&\fz
\end{matrix}\right]\nnh
=\nnh
{\left[\begin{matrix}
\ly_1\w&\ly_2\w&\ly_3\w&\ly_4\w\cr
0&0&0&0\end{matrix}\right]}_{\phantom{j_j}}\hskip-8pt,\\
\hskip17pt\mathrm{while\ }
\left[\begin{matrix}
\mu_+\w\!&\!-\mu_-\w\cr
-\mu_-\w\!&\!\mu_+\w\end{matrix}\right]\nnh\nnh\left[\begin{matrix}
\mu_+\w&\mu_-\w\cr
\mu_-\w&\mu_+\w\end{matrix}\right]\nnh
=\nnh\left[\begin{matrix}
8\ly\sy&0\cr
0&8\ly\sy\end{matrix}\right]\nnh\nnh,
\end{array}
\]
(\ref{sol}-b) and (\ref{sol}-c) follow. Now (\ref{nel}-i) -- 
(\ref{nel}-iv) yield (\ref{sol}-a).
\end{proof}
Due to (\ref{cfo}) we have
$\,\nabla\hn\xi^k\nh=-\vg_{\hskip-2.2ptj}^k\nnh\nh\otimes\xi^j\nh$, that is,
\[
\begin{array}{ll}
\nabla\nnh\aj=-\nh\ej\otimes\bj-\fj\otimes\cj-\gj\otimes\dj,
&\nabla\nh\bj=\ej\otimes\aj+\gj\otimes\cj-\fj\otimes\dj,\\
\nabla\cj=\fj\otimes\aj-\gj\otimes\bj-\hj\otimes\dj,
&\nabla\nh\dj=\gj\otimes\aj+\fj\otimes\bj+\hj\otimes\cj,
\end{array}
\]
where only the first equality needs to be verified, using (\ref{mtx}) --
(\ref{tfr}), as (\ref{swp}) then yields the other three. 
Consequently, for $\,\zeta,\eta,\theta\,$ in (\ref{egf}),
\begin{equation}\label{lmn}
\nabla\zeta=2\hh\gj\nh\otimes\eta-2\fj\nnh\otimes\theta,\,\,\,
\nabla\eta=-\nh2\hh\gj\nh\otimes\zeta+\lj\otimes\theta,\,\,\,
\nabla\theta=2\fj\nnh\otimes\zeta-\lj\otimes\eta.
\end{equation}
\begin{rem}\label{lacst}Conjecture {\rm(\ref{cnj})} is true under the
additional assumption that $\,\ly$ is constant. Namely, one then easily
verifies the italicized statement following (\ref{evp}): having
(\ref{evp}) with $\,d\ly=0\,$ at points where
$\,\ly\sy\ne0\,$ would, by (\ref{sol}-b) --
(\ref{sol}-c), give $\,\fj\nh=\gj=0$, and hence $\,\sy=0$, due to the
equations involving $\,d\fj\hs$ and $\,d\hh\gj\,$ in (\ref{evp}).
\end{rem}

\section{The thir\-ty-six di\-rec\-tion\-al-de\-riv\-a\-tive
equations}\label{ts}
\setcounter{equation}{0}
Assuming (\ref{frm}) -- (\ref{rot})
with $\,\nabla\nnh J=0$, denoting  by $\,\sj\nh_i\w,\ly_i\w,\sy\nnh_i\w$ 
the component functions of $\,\sj,\,d\ly,\,d\sy\,$ relative to 
the local frame $\,e\hn_1\w,e_2\w,e_3\w,e\nh_4\w$, and by
$\,\sj\nh_{ij}\w,\ly\hh_{ij}\w,\sy\nnh_{ij}\w$ their
$\,e\nh_j\w\nh$-di\-rec\-tion\-al derivatives, at points
where $\,\ly\sy\ne0\,$ one has
\begin{enumerate}
\item[{\rm($a$)}] $2\hh[2(\ly_{12}\w-\ly_{21}\w)+\ly_1\w\sj\hn_1\w
+\ly_2\w\sj\hn_2\w]\hh\sy=\ly_1\w\sy\nnh_2\w-\ly_2\w\sy\nnh_1\w$,
\item[{\rm($a_4$)}] $2\hh[2(\ly_{34}\w-\ly_{43}\w)+\ly_3\w\sj\hn_3\w
+\ly_4\w\sj\nnh_4\w]\hh\sy=\ly_3\w\sy\hskip-1.5pt_4\w-\ly_4\w\sy\nnh_3\w$,
\item[{\rm($b$)}] $2\hh[2(\ly_{13}\w-\ly_{31}\w)
+\ly_2\w\sj\hn_3\w-\ly_4\w\sj\hn_1\w]\hh\sy=-\ly_1\w\ly_3\w
-\ly_2\w\sy\hskip-1.5pt_4\w
+(\sy\nnh_2\w-\ly_2\w)\ly_4\w$,
\item[{\rm($b_1$)}] $2\hh[2(\ly_{23}\w-\ly_{32}\w)-\ly_1\w\sj\hn_3\w
-\ly_4\w\sj\hn_2\w]\hh\sy
=\ly_2\w\ly_3\w-(\sy\nnh_1\w+\ly_1\w)\ly_4\w+\ly_1\w\sy\hskip-1.5pt_4\w$,
\item[{\rm($b_2$)}] $2\hh[2(\ly_{14}\w-\ly_{41}\w)+\ly_2\w\sj\nnh_4\w
+\ly_3\w\sj\hn_1\w]\hh\sy=-\ly_2\w\ly_3\w+\ly_2\w\sy\nnh_3\w
-\ly_3\w\sy\nnh_2\w+\ly_1\w\ly_4\w$,
\item[{\rm($b_3$)}] $2\hh[2(\ly_{24}\w-\ly_{42}\w)+\ly_3\w\sj\hn_2\w
-\ly_1\w\sj\nnh_4\w]\hh\sy=-\ly_1\w\ly_3\w-\ly_2\w\ly_4\w-\ly_1\w\sy\nnh_3\w
+\ly_3\w\sy\nnh_1\w$,
\item[{\rm($c$)}] $16\hh[2(\ly_{32}\w-\ly_{41}\w)+\ly_3\w\sj\hn_1\w
+\ly_4\w\sj\hn_2\w]\ly\sy^2\nh\mu_-\w
=8\ly\sy\nh(2\sy+7\ly)(\ly_4\w\sy\nnh_1\w-\ly_3\w\sy\nnh_2\w)$
\item[] \quad$+(3\ly^3\nh+2\ly^2\nh\sy-28\ly\sy^2\nh
-104\sy^3)(\ly_1\w\ly_4\w-\ly_2\w\ly_3\w)$,
\item[{\rm($c_1$)}] $16\hh[2(\ly_{31}\w+\ly_{42}\w)+\ly_4\w\sj\hn_1\w
-\ly_3\w\sj\hn_2\w]\ly\sy^2\nh\mu_+\w
=8\ly\sy\nh(7\ly-2\sy)(\ly_3\w\sy\nnh_1\w+\ly_4\w\sy\nnh_2\w)$
\item[] \quad$+(3\ly^3\nh-2\ly^2\nh\sy-28\ly\sy^2\nh
+104\sy^3)(\ly_1\w\ly_3\w+\ly_2\w\ly_4\w)$,
\item[{\rm($c_4$)}] $16\hh[2(\ly_{23}\w-\ly_{14}\w)-\ly_1\w\sj\hn_3\w
-\ly_2\w\sj\nnh_4\w]\ly\sy^2\nh\mu_+\w
=8\ly\sy\nh(2\sy-7\ly)(\ly_1\w\sy\hskip-1.5pt_4\w-\ly_2\w\sy\nnh_3\w)$
\item[] \quad$+(3\ly^3\nh-2\ly^2\nh\sy-28\ly\sy^2\nh
+104\sy^3)(\ly_2\w\ly_3\w-\ly_1\w\ly_4\w)$,
\item[{\rm($c_5$)}] $16\hh[2(\ly_{13}\w+\ly_{24}\w)
+\ly_2\w\sj\hn_3\w-\ly_1\w\sj\nnh_4\w]\ly\sy^2\nh\mu_-\w
=-8\ly\sy\nh(2\sy+7\ly)(\ly_1\w\sy\nnh_3\w+\ly_2\w\sy\hskip-1.5pt_4\w)$
\item[] \quad$-(3\ly^3\nh+2\ly^2\nh\sy-28\ly\sy^2\nh
-104\sy^3)(\ly_1\w\ly_3\w+\ly_2\w\ly_4\w)$,
\item[{\rm($d$)}] $16\hh[2(\mu_+\w\ly_{11}\w+\mu_-\w\ly_{33}\w)
+\mu_+\w\ly_2\w\sj\hn_1\w+\mu_-\w\ly_4\w\sj\hn_3\w]\ly\sy^2\nh
=-3\mu_*\w(\mu_-\w\ly_2^2+\mu_+\w\ly_4^2)$
\item[] \quad$+8\ly\sy\hh[(2\sy-3\ly)\ly_2\w\sy\nnh_2\w
+(2\sy+3\ly)\ly_4\w\sy\hskip-1.5pt_4\w+4\ly(\ly_1\w\sy\nnh_1\w
-\ly_3\w\sy\nnh_3\w)+32\ly\sy^3]$
\item[] \quad$+4\sy\hh[(\ly^2\nh+4\ly\sy+20\hh\sy^2)\ly_1^2
+(\ly^2\nh-4\ly\sy+20\hh\sy^2)\ly_3^2]$,
\item[{\rm($d_1$)}] $16\hh[2(\mu_-\w\ly_{22}\w+\mu_+\w\ly_{33}\w)
-\mu_-\w\ly_1\w\sj\hn_2\w+\mu_+\w\ly_4\w\sj\hn_3\w]\ly\sy^2\nh
=-3\mu_*\w(\mu_+\w\ly_1^2+\mu_-\w\ly_4^2)$
\item[] \quad$+8\ly\sy\hh[(2\sy+3\ly)\ly_1\w\sy\nnh_1\w
+(2\sy-3\ly)\ly_4\w\sy\hskip-1.5pt_4\w+4\ly(\ly_3\w\sy\nnh_3\w
-\ly_2\w\sy\nnh_2\w)-32\ly\sy^3]$
\item[] \quad$+4\sy\hh[(\ly^2\nh-4\ly\sy+20\hh\sy^2)\ly_2^2
+(\ly^2\nh+4\ly\sy+20\hh\sy^2)\ly_3^2]$,
\item[{\rm($d_2$)}] $16\hh[2(\mu_-\w\ly_{11}\w+\mu_+\w\ly_{44}\w)
+\mu_-\w\ly_2\w\sj\hn_1\w-\mu_+\w\ly_3\w\sj\nnh_4\w]\ly\sy^2\nh
=-3\mu_*\w(\mu_+\w\ly_2^2+\mu_-\w\ly_3^2)$
\item[] \quad$+8\ly\sy\hh[(2\sy+3\ly)\ly_2\w\sy\nnh_2\w
+(2\sy-3\ly)\ly_3\w\sy\nnh_3\w-4\ly(\ly_1\w\sy\nnh_1\w
+\ly_4\w\sy\hskip-1.5pt_4\w)-32\ly\sy^3]$
\item[] \quad$+4\sy\hh[(\ly^2\nh-4\ly\sy+20\hh\sy^2)\ly_1^2
+(\ly^2\nh+4\ly\sy+20\hh\sy^2)\ly_4^2]$,
\item[{\rm($d_3$)}] $16\hh[2(\mu_+\w\ly_{22}\w+\mu_-\w\ly_{44}\w)
-\mu_+\w\ly_1\w\sj\hn_2\w-\mu_-\w\ly_3\w\sj\nnh_4\w]\ly\sy^2\nh
=-3\mu_*\w(\mu_-\w\ly_1^2+\mu_+\w\ly_3^2)$
\item[] \quad$+8\ly\sy\nh[(2\sy-3\ly)\ly_1\w\sy\nnh_1\w
+(2\sy+3\ly)\ly_3\w\sy\nnh_3\w+4\ly(\ly_2\w\sy\nnh_2\w
-\ly_4\w\sy\hskip-1.5pt_4\w)+32\ly\sy^3]$
\item[] \quad$+4\sy\nh[(\ly^2\nh+4\ly\sy+20\hh\sy^2)\ly_2^2
+(\ly^2\nh-4\ly\sy+20\hh\sy^2)\ly_4^2]$,
\item[{\rm($e$)}] $16\hh[2(\mu_+\w\ly_{21}\w+\mu_-\w\ly_{34}\w)
-\mu_+\w\ly_1\w\sj\hn_1\w+\mu_-\w\ly_4\w\sj\nnh_4\w]\ly\sy^2\nh
=32\ly^2\nh\sy(\ly_2\w\sy\nnh_1\w-\ly_3\w\sy\hskip-1.5pt_4\w)$
\item[] \quad$+(3\ly^3\nh-2\ly^2\nh\sy+4\ly\sy^2\nh+104\sy^3)\ly_1\w\ly_2\w
-(3\ly^3\nh+2\ly^2\nh\sy+4\ly\sy^2\nh-104\sy^3)\ly_3\w\ly_4\w$
\item[] \quad$+8\ly\sy\nh[(3\ly-2\sy)\ly_1\w\sy\nnh_2\w
-(3\ly+2\sy)\ly_4\w\sy\nnh_3\w]$,
\item[{\rm($e_1$)}] $16\hh[2(\mu_-\w\ly_{12}\w-\mu_+\w\ly_{34}\w)
+\mu_-\w\ly_2\w\sj\hn_2\w-\mu_+\w\ly_4\w\sj\nnh_4\w]\ly\sy^2\nh
=-32\ly^2\nh\sy(\ly_1\w\sy\nnh_2\w+\ly_3\w\sy\hskip-1.5pt_4\w)$
\item[] \quad$-(3\ly^3\nh+2\ly^2\nh\sy+4\ly\sy^2\nh-104\sy^3)\ly_1\w\ly_2\w
-(3\ly^3\nh-2\ly^2\nh\sy+4\ly\sy^2\nh+104\sy^3)\ly_3\w\ly_4\w$
\item[] \quad$+8\ly\sy\nh[(2\sy-3\ly)\ly_4\w\sy\nnh_3\w
-(2\sy+3\ly)\ly_2\w\sy\nnh_1\w]$,
\item[{\rm($e_2$)}] $16\hh[2(\mu_-\w\ly_{21}\w-\mu_+\w\ly_{43}\w)
-\mu_-\w\ly_1\w\sj\hn_1\w+\mu_+\w\ly_3\w\sj\hn_3\w]\ly\sy^2\nh
=-32\ly^2\nh\sy(\ly_2\w\sy\nnh_1\w+\ly_4\w\sy\nnh_3\w)$
\item[] \quad$-(3\ly^3\nh+2\ly^2\nh\sy+4\ly\sy^2\nh-104\sy^3)\ly_1\w\ly_2\w
-(3\ly^3\nh-2\ly^2\nh\sy+4\ly\sy^2\nh+104\sy^3)\ly_3\w\ly_4\w$
\item[] \quad$+8\ly\sy\nh[(2\sy-3\ly)\ly_3\w\sy\hskip-1.5pt_4\w
-(2\sy+3\ly)\ly_1\w\sy\nnh_2\w]$,
\item[{\rm($e_3$)}]$16\hh[2(\mu_+\w\ly_{12}\w+\mu_-\w\ly_{43}\w)
+\mu_+\w\ly_2\w\sj\hn_2\w-\mu_-\w\ly_3\w\sj\nh_3\w]\ly\sy^2\nh
=32\ly^2\nh\sy(\ly_1\w\sy\nnh_2\w-\ly_4\w\sy\nnh_3\w)$
\item[] \quad$+(3\ly^3\nh-2\ly^2\nh\sy+4\ly\sy^2\nh+104\sy^3)\ly_1\w\ly_2\w
-(3\ly^3\nh+2\ly^2\nh\sy+4\ly\sy^2\nh-104\sy^3)\ly_3\w\ly_4\w$
\item[] \quad$+8\ly\sy\nh[(3\ly-2\sy)\ly_2\w\sy\nnh_1\w
-(3\ly+2\sy)\ly_3\w\sy\hskip-1.5pt_4\w]$,
\item[{\rm($f$)}] $16\hh[2(\sy\nnh_{12}\w-\sy\nnh_{21}\w)+\sy\nnh_1\w\sj\hn_1\w
+\sy\nnh_2\w\sj\hn_2\w]\ly\sy^2\nh
=\mu_*\w(\ly_1\w\sy\nnh_2\w-\ly_2\w\sy\nnh_1\w)
+16\hh\sy^2\nh(\ly_4\w\sy\nnh_3\w-\ly_3\w\sy\hskip-1.5pt_4\w)$,
\item[{\rm($f\nnh_4$)}] $16\hh[2(\sy\hskip-2pt_{43}\w-\sy\nnh_{34}\w)
-\sy\nnh_3\w\sj\hn_3\w-\sy\hskip-1.5pt_4\w\sj\nnh_4\w]\ly\sy^2\nh
=16\hh\sy^2\nh(\ly_1\w\sy\nnh_2\w-\ly_2\w\sy\nnh_1\w)
+\mu_*\w(\ly_4\w\sy\nnh_3\w-\ly_3\w\sy\hskip-1.5pt_4\w)$,
\item[{\rm($g$)}] $16\hh[2(\sy\nnh_{13}\w-\sy\nnh_{31}\w)+\sy\nnh_2\w\sj\hn_3\w
-\sy\hskip-1.5pt_4\w\sj\hn_1\w]\ly\sy^2$
\item[] \quad$=4\sy(\mu_-\w\ly_3\w\sy\nnh_1\w-\mu_+\w\ly_1\w\sy\nnh_3\w)
+\mu_-^2\ly_2\w\sy\hskip-1.5pt_4\w-\mu_+^2\ly_4\w\sy\nnh_2\w$,
\item[{\rm($g_1$)}] $16\hh[2(\sy\nnh_{32}\w-\sy\nnh_{23}\w)
+\sy\nnh_1\w\sj\hn_3\w+\sy\hskip-1.5pt_4\w\sj\hn_2\w]\ly\sy^2$
\item[] \quad$=\mu_+^2\ly_1\w\sy\hskip-1.5pt_4\w-\mu_-^2\ly_4\w\sy\nnh_1\w
  +4\sy(\mu_-\w\ly_2\w\sy\nnh_3\w
-\mu_+\w\ly_3\w\sy\nnh_2\w)$,
\item[{\rm($g_2$)}] $16\hh[2(\sy\hskip-2pt_{41}\w-\sy\nnh_{14}\w)
-\sy\nnh_2\w\sj\nnh_4\w-\sy\nnh_3\w\sj\hn_1\w]\ly\sy^2$
\item[] \quad$=\mu_+^2\ly_2\w\sy\nnh_3\w-\mu_-^2\ly_3\w\sy\nnh_2\w
+4\sy\nh(\mu_-\w\ly_1\w\sy\hskip-1.5pt_4\w-\mu_+\w\ly_4\w\sy\nnh_1\w)$,
\item[{\rm($g_3$)}] $16\hh[2(\sy\hskip-2pt_{42}\w-\sy\nnh_{24}\w)
+\sy\nnh_1\w\sj\nnh_4\w-\sy\nnh_3\w\sj\hn_2\w]\ly\sy^2$
\item[] \quad$=\mu_+^2\ly_3\w\sy\nnh_1\w-\mu_-^2\ly_1\w\sy\nnh_3\w
+4\sy\nh(\mu_+\w\ly_2\w\sy\hskip-1.5pt_4\w-\mu_-\w\ly_4\w\sy\nnh_2\w)$,
\item[{\rm($h$)}] $16\hh[2(\ly_{11}\w+\ly_{22}\w)
+\ly_2\w\sj\hn_1\w-\ly_1\w\sj\hn_2\w]\ly\sy^2\nh\mu_*\w$
\item[] \quad$-128\hh[2(\sy\nnh_{11}\w+\sy\nnh_{22}\w)
+\sy\nnh_2\w\sj\hn_1\w-\sy\nnh_1\w\sj\hn_2\w]\ly^2\nh\sy^3\nh
+320\ly^2\nnh\sy^2\nh(\sy_1^2+\sy_2^2)$
\item[] \quad$=(112\sy^4\nh+40\ly^2\nnh\sy^2\nh-\ly^4)(\ly_1^2+\ly_2^2)
-48\sy^2\nnh\mu_*\w(\ly_3^2+\ly_4^2)$
\item[] \quad$+16\ly\sy[(4\sy^2\nh-5\ly^2)(\ly_1\w\sy\nnh_1\w
+\ly_2\w\sy\nnh_2\w)
+8\sy^2\nh(\ly_3\w\sy\nnh_3\w+\ly_4\w\sy\hskip-1.5pt_4\w)+32\ly^2\nnh\sy^3]$,
\item[{\rm($h_4$)}] $16\hh[2(\ly_{33}\w+\ly_{44}\w)+\ly_4\w\sj\hn_3\w
-\ly_3\w\sj\hn_4\w]\ly\sy^2\nh\mu_*\w$
\item[] \quad$-128\hh[2(\sy\nnh_{33}\w+\sy\hskip-2pt_{44}\w)
+\sy\hskip-1.5pt_4\w\sj\hn_3\w-\sy\nnh_3\w\sj\hn_4\w]\ly^2\nh\sy^3\nh
+320\ly^2\nnh\sy^2\nh(\sy_3^2+\sy\hskip-1.5pt_4^2)$
\item[] \quad$=(112\sy^4\nh+40\ly^2\nnh\sy^2\nh-\ly^4)(\ly_3^2+\ly_4^2)
-48\sy^2\nnh\mu_*\w(\ly_1^2+\ly_2^2)$
\item[] \quad$+16\ly\sy[(4\sy^2\nh-5\ly^2)(\ly_3\w\sy\nnh_3\w
+\ly_4\w\sy\hskip-1.5pt_4\w)
+8\sy^2\nh(\ly_1\w\sy\nnh_1\w+\ly_2\w\sy\nnh_2\w)-32\ly^2\nnh\sy^3]$,
\item[{\rm($i$)}] $16\hh[2(\ly_{41}\w+\ly_{23}\w)
-\ly_1\w\sj\hn_3\w-\ly_3\w\sj\hn_1\w]\ly\sy^2\nh\mu_*\w$
\item[] \quad$-128\hh[2(\sy\hskip-2pt_{41}\w+\sy\nnh_{23}\w)
-\sy\nnh_1\w\sj\hn_3\w-\sy\nnh_3\w\sj\hn_1\w]\ly^2\nh\sy^3\nh$
\item[] \quad$=(304\sy^4\nh-160\ly\sy^3\nh+56\ly^2\nnh\sy^2\nh+8\ly^3\nnh\sy
-\ly^4)\ly_2\w\ly_3\w$
\item[] \quad$+\,\hs(304\sy^4\nh+160\ly\sy^3\nh+56\ly^2\nnh\sy^2\nh
-8\ly^3\nnh\sy-\ly^4)\ly_1\w\ly_4\w$
\item[] \quad$-8\ly\sy[\mu_-^2\ly_3\w\sy\nnh_2\w
+\mu_+^2\ly_1\w\sy\hskip-1.5pt_4\w
+40\ly\sy(\sy\nnh_1\w\sy\hskip-1.5pt_4\w+\sy\nnh_2\w\sy\nnh_3\w)]$
\item[] \quad$-8\ly\sy[(4\sy^2\nh-4\ly\sy+9\ly^2)\ly_2\w\sy\nnh_3\w
+(4\sy^2\nh+4\ly\sy+9\ly^2)\ly_4\w\sy\nnh_1\w]$,
\item[{\rm($i_1$)}] $16\hh[2(\ly_{42}\w-\ly_{13}\w)
-\ly_2\w\sj\hn_3\w-\ly_3\w\sj\hn_2\w]\ly\sy^2\nh\mu_*\w$
\item[] \quad$-128\hh[2(\sy\hskip-2pt_{42}\w-\sy\nnh_{13}\w)
-\sy\nnh_2\w\sj\hn_3\w-\sy\nnh_3\w\sj\hn_2\w]\ly^2\nh\sy^3\nh$
\item[] \quad$=(304\sy^4\nh-160\ly\sy^3\nh+56\ly^2\nnh\sy^2\nh+8\ly^3\nnh\sy
-\ly^4)\ly_2\w\ly_4\w$
\item[] \quad$-\,\hs(304\sy^4\nh+160\ly\sy^3\nh+56\ly^2\nnh\sy^2\nh
-8\ly^3\nnh\sy-\ly^4)\ly_1\w\ly_3\w$
\item[] \quad$+8\ly\sy[\mu_+^2\ly_3\w\sy\nnh_1\w
-\mu_-^2\ly_2\w\sy\hskip-1.5pt_4\w
+40\ly\sy(\sy\nnh_1\w\sy\nnh_3\w-\sy\nnh_2\w\sy\hskip-1.5pt_4\w)]$
\item[] \quad$+8\ly\sy[(4\sy^2\nh+4\ly\sy+9\ly^2)\ly_1\w\sy\nnh_3\w
-(4\sy^2\nh-4\ly\sy+9\ly^2)\ly_4\w\sy\nnh_2\w]$,
\item[{\rm($i_2$)}] $16\hh[2(\ly_{24}\w-\ly_{31}\w)
-\ly_1\w\sj\nnh_4\w-\ly_4\w\sj\hn_1\w]\ly\sy^2\nh\mu_*\w$
\item[] \quad$-128\hh[2(\sy\nnh_{24}\w-\sy\nnh_{31}\w)
-\sy\nnh_1\w\sj\nnh_4\w-\sy\hskip-1.5pt_4\w\sj\hn_1\w]\ly^2\nh\sy^3\nh$
\item[] \quad$=(304\sy^4\nh+160\ly\sy^3\nh+56\ly^2\nnh\sy^2\nh-8\ly^3\nnh\sy
-\ly^4)\ly_2\w\ly_4\w$
\item[] \quad$-\,\hs(304\sy^4\nh-160\ly\sy^3\nh+56\ly^2\nnh\sy^2\nh
+8\ly^3\nnh\sy-\ly^4)\ly_1\w\ly_3\w$
\item[] \quad$+8\ly\sy[\mu_-^2\ly_1\w\sy\nnh_3\w
-\mu_+^2\ly_4\w\sy\nnh_2\w
+40\ly\sy(\sy\nnh_1\w\sy\nnh_3\w-\sy\nnh_2\w\sy\hskip-1.5pt_4\w)]$
\item[] \quad$+8\ly\sy[(4\sy^2\nh-4\ly\sy+9\ly^2)\ly_3\w\sy\nnh_1\w
-(4\sy^2\nh+4\ly\sy+9\ly^2)\ly_2\w\sy\hskip-1.5pt_4\w]$,
\item[{\rm($i_3$)}] $16\hh[2(\ly_{14}\w+\ly_{32}\w)
+\ly_2\w\sj\nnh_4\w+\ly_4\w\sj\hn_2\w]\ly\sy^2\nh\mu_*\w$
\item[] \quad$-128\hh[2(\sy\nnh_{14}\w+\sy\nnh_{32}\w)
+\sy\nnh_2\w\sj\nnh_4\w+\sy\hskip-1.5pt_4\w\sj\hn_2\w]\ly^2\nh\sy^3\nh$
\item[] \quad$=(304\sy^4\nh-160\ly\sy^3\nh+56\ly^2\nnh\sy^2\nh+8\ly^3\nnh\sy
-\ly^4)\ly_1\w\ly_4\w$
\item[] \quad$+\,\hs(304\sy^4\nh+160\ly\sy^3\nh+56\ly^2\nnh\sy^2\nh
-8\ly^3\nnh\sy-\ly^4)\ly_2\w\ly_3\w$
\item[] \quad$-8\ly\sy[\mu_-^2\ly_4\w\sy\nnh_1\w
+\mu_+^2\ly_2\w\sy\nnh_3\w
+40\ly\sy(\sy\nnh_2\w\sy\nnh_3\w+\sy\nnh_1\w\sy\hskip-1.5pt_4\w)]$
\item[] \quad$-8\ly\sy[(4\sy^2\nh-4\ly\sy+9\ly^2)\ly_1\w\sy\hskip-1.5pt_4\w
+(4\sy^2\nh+4\ly\sy+9\ly^2)\ly_3\w\sy\nnh_2\w]$,
\item[{\rm($j$)}] $16\hh[2(\sj\hn_{21}\w-\sj\hn_{12}\w)-\sj_1^2
-\sj_2^2]\ly\sy^2\nh+8\ly\sy\nh(\sy\nnh_2\w\sj\hn_1\w
-\sy\nnh_1\w\sj\hn_2\w)$
\item[] \quad$=\mu_*\w(\ly_2\w\sj\hn_1\w-\ly_1\w\sj\hn_2\w)
+16\hh\sy^2\nh(\ly_3\w\sj\nnh_4\w-\ly_4\w\sj\hn_3\w)-32\ly^2\nnh\sy^2\nh$,
\item[{\rm($j_4$)}] $16\hh[2(\sj\nh_{43}\w-\sj\hn_{34}\w)-\sj_3^2
-\sj_4^2]\ly\sy^2\nh+8\ly\sy\nh(\sy\hskip-1.5pt_4\w\sj\hn_3\w
-\sy\nnh_3\w\sj\nnh_4\w)$
\item[] \quad$=\mu_*\w(\ly_4\w\sj\hn_3\w-\ly_3\w\sj\nnh_4\w)
+16\hh\sy^2\nh(\ly_1\w\sj\hn_2\w-\ly_2\w\sj\hn_1\w)
+32\ly^2\nnh\sy^2\nh$,
\item[{\rm($k$)}] $16\hh[2(\sj\hn_{31}\w-\sj\hn_{13}\w)
+\sj\hn_1\w\sj\nnh_4\w-\sj\hn_2\w\sj\hn_3\w]\ly\sy^2\nh
+8\ly\sy\nh(\sy\nnh_2\w\sj\nnh_4\w-\sy\hskip-1.5pt_4\w\sj\hn_2\w)$
\item[] \quad$=\mu_+^2\ly_4\w\sj\hn_2\w-\mu_-^2\ly_2\w\sj\nnh_4\w
+4\sy\nh(\mu_+\w\ly_1\w\sj\hn_3\w-\mu_-\w\ly_3\w\sj\hn_1\w)$,
\item[{\rm($k_1$)}] $16\hh[2(\sj\hn_{32}\w-\sj\hn_{23}\w)
+\sj\hn_2\w\sj\nnh_4\w+\sj\hn_1\w\sj\hn_3\w]\ly\sy^2\nh
+8\ly\sy\nh(\sy\hskip-1.5pt_4\w\sj\hn_1\w-\sy\nnh_1\w\sj\nnh_4\w)$
\item[] \quad$=\mu_+^2\ly_1\w\sj\nnh_4\w-\mu_-^2\ly_4\w\sj\hn_1\w
+4\sy\nh(\mu_-\w\ly_2\w\sj\hn_3\w-\mu_+\w\ly_3\w\sj\hn_2\w)$,
\item[{\rm($k_2$)}] $16\hh[2(\sj\nnh_{41}\w-\sj\hn_{14}\w)
-\sj\hn_1\w\sj\hn_3\w-\sj\hn_2\w\sj\nnh_4\w]\ly\sy^2\nh
+8\ly\sy\nh(\sy\nnh_3\w\sj\hn_2\w-\sy\nnh_2\w\sj\hn_3\w)$
\item[] \quad$=\mu_+^2\ly_2\w\sj\hn_3\w-\mu_-^2\ly_3\w\sj\hn_2\w
+4\sy\nh(\mu_-\w\ly_1\w\sj\nnh_4\w-\mu_+\w\ly_4\w\sj\hn_1\w)$,
\item[{\rm($k_3$)}] $16\hh[2(\sj\nnh_{42}\w-\sj\hn_{24}\w)
-\sj\hn_2\w\sj\hn_3\w+\sj\hn_1\w\sj\nnh_4\w]\ly\sy^2\nh
+8\ly\sy\nh(\sy\nnh_1\w\sj\hn_3\w-\sy\nnh_3\w\sj\hn_1\w)$
\item[] \quad$=\mu_+^2\ly_3\w\sj\hn_1\w-\mu_-^2\ly_1\w\sj\hn_3\w
+4\sy\nh(\mu_+\w\ly_2\w\sj\nnh_4\w-\mu_-\w\ly_4\w\sj\hn_2\w)$.
\end{enumerate}
In fact, as $\,(2\ej\nh,\,2\hj)=(\sj+\lj,\,\sj-\lj)$, 
(\ref{mtx}) yields the Lie-brack\-et relations
\begin{equation}\label{lbk}
\begin{array}{l}
2[\hh e\hn_1\w,e_2\w]=(\sj\hn_1\w+\lj_1\w)\hh e\hn_1\w
+(\sj\hn_2\w+\lj_2\w)\hh e_2\w+2(\gj\nh_1\w+\fz)\hh e_3\w
+2(\gj\nh_2\w-\fe)\hh e\nh_4\w\hh,\\
2[\hh e\hn_1\w,e_3\w]=2\fe e\hn_1\w
+(\sj\hn_3\w+\lj_3\w-2\hh\gj\hn_1\w)\hh e_2\w
+2\fd e_3\w+(2\gj\nh_3\w+\lj_1\w-\sj\hn_1\w)\hh e\nh_4\w\hh.
\end{array}
\end{equation}
Taking the directional derivatives of $\,\ly\,$ and $\,\sy\,$ along the
vector fields (\ref{lbk}) multiplied by $\,16\ly\sy^2$, and using
(\ref{sol}), which, in particular, gives
\begin{equation}\label{inp}
8\ly\sy(\gj\nh_1\w\nh+\fz)=-4\sy\nnh\ly_4\w,\qquad
8\ly\sy(\gj\nh_2\w-\fe)=4\sy\nnh\ly_3\w,
\end{equation}
we obtain ($a$), ($b$), ($f$) and ($g$). Next, given a $\,1$-form
$\,Z=Z_i\w e^i\nh$, let $\,Z_{ij}\w=d_j\w Z_i\w$. By (\ref{tfr}) and
(\ref{evp}), $\,d\hskip.2ptZ\,$ is the combination of 
$\,\aj\wedge\bj,\,\aj\wedge\cj,\,\aj\wedge\nh\dj,\,\bj\wedge\cj,\,
\bj\wedge\hn\dj,\,\cj\wedge\nh\dj$ with the first three respective
coefficient functions
\[
\begin{array}{l}
Z_{21}\w-Z_{12}\w-Z_1\w\ej\hn_1\w-Z_2\w\ej\hn_2\w-Z_3\w(\fz+\gj\nh_1\w)
+Z_4\w(\fe-\gj\nh_2\w),\\
Z_{31}\w-Z_{13}\w-Z_1\w\fe+Z_2\w(\gj\nh_1\w-\ej\hn_3\w)
-Z_3\w\fd+Z_4\w(\hj\nnh_1\w-\gj\nh_3\w),\\
Z_{41}\w-Z_{14}\w-Z_1\w\gj\nh_1\w-Z_2\w(\fe+\ej\hn_4\w)
-Z_3\w(\fv+\hj\nnh_1\w)-Z_4\w\gj\nnh_4\w.
\end{array}
\]
In view of (\ref{evp}), the above three lines equal
\[
\begin{array}{l}
\lj_1\w\gj\nh_2\w-\lj_2\w\gj\nh_1\w,\qquad
\lj_1\w\gj\nh_3\w-\lj_3\w\gj\nh_1\w,\qquad
\lj_1\w\gj\nh_4\w-\lj_4\w\gj\nnh_1\w\,\mathrm{\ \ when\ }\,Z=\fj\nh, \\
4(\fz\gj\nh_1\w-\fe\gj\nh_2\w)-\ly,\,\,\,
4(\fd\gj\nh_1\w-\fe\gj\nh_3\w),\,\,\,
4(\fv\gj\nh_1\w-\fe\gj\nnh_4\w)\,\mathrm{\,\ if\ }\,Z=\hs\lj.
\end{array}
\]
For $\,Z=S\,$ the first two of them are $\,-\nnh\ly\,$ and $\,0$. Of the
eight equalities just described, multiplying the first three by
$\,256\ly^2\nh\sy^3\nh$, the next two by $\,512\ly^2\nnh\sy^4\nh$,
the last two by $\,32\ly\sy^2\nh$, then using (\ref{sol}), 
while replacing $\,\ej\nh$ by $\,(\sj+\lj)/2\,$ and $\,\hj$ by
$\,(\sj-\lj)/2$, we get  
($c$), ($d$), ($e$), ($h$), ($i$), ($j$), ($k$). The equalities labeled by
plain letters $\,a,b,\dots,j,k\,$ have now been established. They
imply the remaining ones, with subscripts $\,1,2,3,4,5\,$ in their labels,
each subscript referring to one of the five cases of (\ref{rpl}). Thus, 
($e_2$) arises from ($e$) via (\ref{rpl}-ii), and ($h_4$) from ($h$) via
(\ref{rpl}-iv).
\begin{rem}\label{cmbin}One easily verifies that equation ($e_3$), or ($e_1$),
or ($i_1$), or ($i_3$), is the linear combination of ($a$), ($a_4$), ($e$),
or ($a$), ($a_4$), ($e_2$), or ($c_1$), ($c_5$), ($g$), ($g_3$), ($i_2$), or 
($c$), ($c_4$), ($g_2$), ($g_1$), ($i$), with the respective coefficient
functions 
$\,8\ly\sy\nh\mu_+\w,\,-8\ly\sy\nh\mu_-\w,\,1$, or
$\,8\ly\sy\nh\mu_-\w,\,-8\ly\sy\nh\mu_+\w,\,1$, or
$\,\mu_-\w,\,-\mu_+\w,\,8\ly\sy,\,-8\ly\sy,1\,$ or, finally, 
$\,\mu_+\w,\,-\mu_-\w,\,8\ly\sy,\,-8\ly\sy,1$. On the other hand, 
the linear combination of ($b_2$), ($b_1$), ($c$), ($c_4$) with the
coefficients 
$\,8\ly\sy\nh\mu_*\w,\,-8\ly\sy\nh\mu_*\w,\,-\mu_+\w,\,\mu_-\w$ yields
$\,32\ly\sy\,$ times equation (a) in the Introduction, while (b) then 
follows from (a) via (\ref{rep}-ii). Equivalently, (b) multiplied by
$\,32\ly\sy\,$ also arises as the combination of ($b$), ($b_3$), ($c_1$),
($c_5$) with the coefficients 
$\,8\ly\sy\nh\mu_*\w,\,8\ly\sy\nh\mu_*\w,\,\mu_-\w,\,-\mu_+\w$.
\end{rem}
\begin{rem}\label{intgr}Under the assumptions (\ref{frm}) -- (\ref{rot}), at
points where $\,\ly\sy\ne0$, due to (\ref{nzc}), the Ric\-ci
eigen\-dis\-tri\-bu\-tions are $\,\mathrm{span}\hh(e\hn_1\w,e_2\w)\,$ and 
$\,\mathrm{span}\hh(e_3\w,e\nh_4\w)$. The former (or, the latter) 
is in\-te\-gra\-ble if and only if $\,\ly_3\w=\ly_4\w=\hs0\,$ (or,
respectively, $\,\ly_1\w=\ly_2\w=\hs0$). In fact, the first claim is obvious
from the first line of (\ref{lbk}) and (\ref{inp}), and it implies the
second one as a consequence of (\ref{rep}-iv).
\end{rem}

\section{Generic points}\label{gp}
\setcounter{equation}{0}
To prove an equality-type conclusion ($*$) on a manifold $\,M\nh$, one
can obviously fix a tensor field $\,\varTheta\nh$, with the zero set
$\,Y\nnh\nh$, and establish ($*$) just on the open dense set
$\,Y\nh^o\nnh\cup[M\nnh\smallsetminus Y]$ of
$\,\varTheta\nh${\it-ge\-ner\-ic points}, arising as the union of 
the interior $\,Y\nh^o$ of $\,Y\nnh$ and the complement
$\,M\nnh\smallsetminus Y\nnh$ of $\,Y\nnh\nh$.

Our proofs of Conjecture (\ref{cnj}) in the special cases (ii) -- (iii) of
Theorem~\ref{thrsc} will consist in showing that, under either of the
assumptions (ii) -- (iii),
\begin{equation}\label{lsz}
\ly\sy=0\,\mathrm{\ identically\ whenever\
(\ref{frm})\,}-\mathrm{\,(\ref{rot})\ \ hold\ and\ }\,\nabla\nnh J=0,
\end{equation}
which is a slightly modified version of (\ref{stt}).

In fact, suppose that we have established the equality $\,\ly\sy=0\,$ in
(\ref{lsz}). At $\,\sy\nh$-ge\-ner\-ic points, with $\,Y\nnh$ as above for
$\,\varTheta\nh=\sy$, the metric restricted to $\,Y\nh^o\nnh$, or
$\,M\nnh\smallsetminus Y\nnh$, is, due to (\ref{nzc}),
con\-for\-mal\-ly flat or, respectively,
Ric\-ci-flat, and so  -- see (\ref{cff}) -- on both 
sets the Ric\-ci tensor
is parallel. This is the assertion of (\ref{cnj}). It also shows that
either $\,\ly=0\,$ identically, or $\,\ly\ne0=\sy\,$ everywhere, as required
in (\ref{stt}).

In both cases (ii), (iii), we will derive a contradiction from the
assumption that
\begin{equation}\label{wwa}
\mathrm{(\ref{frm})\,}-\mathrm{\,(\ref{rot})\ \ with\ }\,\nabla\nnh J
=0\,\mathrm{\ and\ }\,\ly\sy\ne0\,\mathrm{\ everywhere.}
\end{equation}
In other words, we will consider the zero set $\,Y\nnh$ of
$\,\ly\sy\,$ and the $\,\ly\sy\nh$-ge\-ner\-ic set 
$\,Y\nh^o\nnh\cup[M\nnh\smallsetminus Y]$. On $\,Y\nh^o$ our claim,
$\,\ly\sy=0\,$ in (\ref{lsz}), holds trivially, and the contradiction
resulting from (\ref{wwa}), where we have replaced $\,M\,$ by
$\,M\nnh\smallsetminus Y\nnh$, amounts to 
showing that $\,M\nnh\smallsetminus Y\hs$ is empty.

In addition to (\ref{wwa}), we may -- and will -- assume that
\begin{equation}\label{lfz}
\ly_4\w\,=\,\hs0\quad\mathrm{and}\quad\ly_3\w>\hh0\quad\mathrm{everywhere,}
\end{equation}
which results in no loss of generality:
with $\,Y\nnh$ being this time the zero set of $\,d\ly$, Remark~\ref{lacst}
yields our assertion, $\,\ly\sy=0$, on $\,Y\nh^o\nnh$, while in 
$\,M\nnh\smallsetminus Y\nnh$, locally, (\ref{rep}-iv) and Lemma~\ref{cnsym}
allow us, respectively, to require that $\,\nabla\nnh\ly\,$ have a nonzero 
projection onto $\,\mathrm{span}\hh(e_3\w,e\nh_4\w)$, and that this
projection be equal to $\,\ly_3\w e_3\w$ with $\,\ly_3\w>\hh0$.
\begin{rem}\label{summr}To summarize: in the next two sections we will prove
assertion of Theorem~\ref{thrsc}, in the form $\,\ly\sy=0\,$ -- cf.\
(\ref{lsz}) -- assuming, in addition, (ii) or (iii). In the latter case,
we will consider the $\,d\ly$-ge\-ner\-ic set  
$\,Y\nh^o\nnh\cup[M\nnh\smallsetminus Y]$, and observe that
$\,\ly\sy=0\,$ on $\,Y\nh^o\nnh$, so that we are free to replace $\,M\,$ by 
$\,M\nnh\smallsetminus Y\nnh$. In both cases, we will then show that the
conditions (\ref{wwa}) -- (\ref{lfz}) lead to a contradiction.
\end{rem}

\section{Proof of case {\rm(ii)} in Theorem~\ref{thrsc}}\label{pc}
\setcounter{equation}{0}
Following Remark~\ref{summr}, we now assume (\ref{wwa}) -- (\ref{lfz}), along
with (ii) in Theorem~\ref{thrsc}, in order to derive a contradiction.

In view of Remark~\ref{intgr} and (\ref{lfz}), 
$\,\mathrm{span}\hh(e\hn_1\w,e_2\w)\,$ cannot be 
in\-te\-gra\-ble. Since we are assuming (ii),
this leads to in\-te\-gra\-bi\-li\-ty\ of
$\,\mathrm{span}\hh(e_3\w,e\nh_4\w)$, with $\,\ly_1\w=\ly_2\w=\hs0$.

As we now have $\,\ly_1\w=\ly_2\w=\ly_4\w=\hs0$, 
($b$), ($b_1$), ($b_2$) and ($b_3$) give, successively,
\begin{equation}\label{suc}
\ly_{31}\w=0,\quad\ly_{32}\w=0,\quad2\sy\nh\sj\hn_1\w=-\sy\nnh_2\w,\quad
2\sy\nh\sj\hn_2\w=\sy\nnh_1\w,
\end{equation}
turning ($c$) and ($c_1$), with $\,\ly\sy\nnh\ly_3\w\ne0$, into
$\,\sy\nnh_1\w=\sy\nnh_2\w=0$, and so, by (\ref{suc}), 
$\,\sj\hn_1\w=\sj\hn_2\w=0$. From 
($j$) and ($h$)
we thus get $\,\ly_3\w\sj\nnh_4\w=2\ly^2$ and
$\,3\mu_*\w\ly_3^2=8\ly\sy\nh(\ly_3\w\sy\nnh_3\w+4\ly^2\nnh\sy)$. 
Combined with ($d_2$), or with ($d_3$), this yields
\begin{equation}\label{lts}
\ly_3\w\sy\nnh_3\w=-4\sy\nnh\mu_*\w\mathrm{,\ so\ that\ }\,3\mu_*\w\ly_3^2
=64\ly\sy^2\nh(\ly^2\nh-2\sy^2).
\end{equation}
Next, ($d$) and ($d_1$), divided by $\,4\sy\nnh$, amount to
\begin{equation}\label{els}
\begin{array}{rl}
\mathrm{i)}&8\ly\sy\nnh\mu_-\w\ly_{33}\w
=8\ly^2\nh(8\sy^3\nh-\ly_3\w\sy\nnh_3\w)
+(\ly^2\nh-4\ly\sy+20\hh\sy^2)\ly_3^2,\\
\mathrm{ii)}&8\ly\sy\nnh\mu_+\w\ly_{33}\w
=8\ly^2\nh(\ly_3\w\sy\nnh_3\w-8\sy^3)+(\ly^2\nh+4\ly\sy+20\hh\sy^2)\ly_3^2.
\end{array}
\end{equation}
The combination of (\ref{els}-i) and (\ref{els}-ii) with the coefficients
$\,3\mu_*\w\mu_+\w$ and $\,-3\mu_*\w\mu_-\w$, evaluated via 
(\ref{lts}), then divided by $\,512\ly^2\nh\sy^2\nh$, reads
$\,0=\ly^4\nh-5\ly^2\nnh\sy^2\nh+12\sy^4\nh$. This is a 
contradiction, the right-hand side being clearly 
positive unless $\,\ly=\sy=\hs0$.

\section{Case {\rm(iii)}}\label{ct}
\setcounter{equation}{0}
As stated in Remark~\ref{summr}, we first consider the $\,d\ly$-ge\-ner\-ic
set $\,Y\nh^o\nnh\cup[M\nnh\smallsetminus Y]$, and note that,
due to Remark~\ref{lacst}, our assertion, $\,\ly\sy=0$, holds on
$\,Y\nh^o\nnh$. We may thus replace $\,M\,$ with 
$\,M\nnh\smallsetminus Y\nnh$, and assume (\ref{wwa}) -- (\ref{lfz}), along 
with (iii) in Theorem~\ref{thrsc}, in order to derive a contradiction.
In view of (\ref{wwa}) -- (\ref{lfz}),
\begin{equation}\label{nbl}
\ly\sy\nh(\ly_1^2+\ly_2^2+\ly_3^2)\ne0\,\,\mathrm{\ everywhere.}
\end{equation}
By (iii) and (\ref{nzc}), 
$\,\sy\,$ is, locally, a function of $\,\ly$, with the derivative
$\,\sy\hh'\nh=d\hh\sy\hn/d\ly$.

Next, according to case (ii) of
Theorem~\ref{thrsc}, already established in Sect.\,\ref{pc},
combined with Remark~\ref{intgr}, on any
nonempty open set, one cannot simultaneously have 
$\,\ly_1\w=\ly_2\w=\hs0\,$ and the condition $\,\ly\sy\ne0\,$ in (\ref{wwa}). 
Thus, if $\,Y\nh^o\nnh\cup[M\nnh\smallsetminus Y]\,$ now denotes 
the $\,(\ly_1^2+\ly_2^2)$-ge\-ner\-ic set, then $\,Y\nh^o$ must be empty. 
As $\,\ly_1\w$ and $\,\ly_2\w$ cannot both vanish on
$\,M\nnh\smallsetminus Y\nnh$ (now replacing $\,M$), 
(a) -- (b) in the Introduction yield
\begin{equation}\label{fsq}
\begin{array}{l}
\mathrm{i)}\hskip7pt8\ly\sy\nh\sy\hh'\nh=12\sy^2\nh-\ly^2\nh,\qquad
\mathrm{ii)}\hskip7pt64\ly^2\nnh\sy^3\nnh\sy\hh''\nh
=(4\sy^2\nh-\ly^2)(12\sy^2\nh+\ly^2),
\end{array}
\end{equation}
while $\,\sy\nnh_i\w=\sy\hh'\nnh\ly_i\w$ and
$\,\sy\nnh_{ij}\w=\sy\hh'\nnh\ly\hh_{ij}\w+\sy\hh''\nnh\ly_i\w\ly_j\w$. 
Hence, replacing
\[
\begin{array}{l}
8\ly\sy\nh\sy\nnh_i\w\mathrm{\ with\ }\,(12\sy^2\nh-\ly^2)\ly_i\w,\qquad
8\ly\sy\nh\sy\nnh_i\w-\mu_*\w\ly_i\w\mathrm{\ with\ }\,8\sy^2\nnh\ly_i\w,\\
8\ly\sy^2\nh(8\ly\sy\nh\sy\nnh_{ij}\w
-\mu_*\w\ly\hh_{ij}\w)\mathrm{\ with\ }\,64\ly\sy^4\nnh\ly\hh_{ij}\w
+(4\sy^2\nh-\ly^2)(12\sy^2\nh+\ly^2)\ly_i\w\ly_j\w,
\end{array}
\]
as allowed by (\ref{fsq}), we rewrite 
($d$), ($d_1$), ($d_2$), ($d_3$),
($h$), ($h_4$) in Sect.\,\ref{ts} as
\begin{enumerate}
\item[] $4\hh[2(\mu_+\w\ly_{11}\w+\mu_-\w\ly_{33}\w)
+\mu_+\w\ly_2\w\sj\hn_1\w]\ly\sy^2\nh
=(\ly-6\hh\sy)\ly\sy\nh\ly_2^2+64\ly^2\nnh\sy^4$
\item[] \hskip40pt$+(20\hh\sy^3\nh+16\ly\sy^2+\ly^2\nnh\sy-\ly^3)\ly_1^2
+(20\hh\sy^3\nh-16\ly\sy^2+\ly^2\nnh\sy+\ly^3)\ly_3^2$,
\item[] $4\hh[2(\mu_-\w\ly_{22}\w+\mu_+\w\ly_{33}\w)
-\mu_-\w\ly_1\w\sj\hn_2\w]\ly\sy^2\nh=(\ly+6\hh\sy)\ly\sy\nh\ly_1^2
-64\ly^2\nnh\sy^4$
\item[] \hskip40pt$+(20\hh\sy^3\nh-16\ly\sy^2+\ly^2\nnh\sy+\ly^3)\ly_2^2
+(20\hh\sy^3\nh+16\ly\sy^2+\ly^2\nnh\sy-\ly^3)\ly_3^2$,
\item[] $4\hh(2\mu_-\w\ly_{11}\w+\mu_-\w\ly_2\w\sj\hn_1\w
-\mu_+\w\ly_3\w\sj\nnh_4\w)\ly\sy^2\nh=\ly\sy\hh[(\ly+6\hh\sy)\ly_2^2
+(\ly-6\hh\sy)\ly_3^2]$
\item[] \hskip137.5pt$+(20\hh\sy^3\nh-16\ly\sy^2+\ly^2\nnh\sy+\ly^3)\ly_1^2
-64\ly^2\nnh\sy^4\nnh$,
\item[] $4\hh(2\mu_+\w\ly_{22}\w-\mu_+\w\ly_1\w\sj\hn_2\w
-\mu_-\w\ly_3\w\sj\nnh_4\w)\ly\sy^2=\ly\sy\hh[(\ly-6\hh\sy)\ly_1^2
+(\ly+6\hh\sy)\ly_3^2]$
\item[] \hskip137.5pt$+(20\hh\sy^3\nh+16\ly\sy^2+\ly^2\nnh\sy-\ly^3)\ly_2^2
+64\ly^2\nnh\sy^4\nnh$,
\item[] $4\hh[2(\ly_{11}\w+\ly_{22}\w)
+\ly_2\w\sj\hn_1\w-\ly_1\w\sj\hn_2\w]\ly\sy^2\nh
=10\hh\sy^2\nh(\ly_1^2+\ly_2^2)
-\ly^2\nnh\ly_3^2-16\ly^3\nnh\sy^2\nh$,
\item[] $4\hh(2\ly_{33}\w-\ly_3\w\sj\nnh_4\w)\ly\sy^2\nh
=10\hh\sy^2\nnh\ly_3^2-\ly^2\nh(\ly_1^2+\ly_2^2)+16\ly^3\nnh\sy^2\nh$.
\end{enumerate}
Subtracting the sum of the last two equations above, multiplied by $\,4\sy$,
from the sum of the first four, we get 
$\,8\ly^2\nnh\sy\nh(\ly_1^2+\ly_2^2+\ly_3^2)=0$, which 
contradicts (\ref{nbl}).



\begin{thebibliography}{99}

\bibitem{arias-marco-kowalski}T.\,Arias-Marco and O.\,Kowalski,
\textit{Classification of $4$-dimensional homogeneous weakly Ein\-stein
 manifolds}, Czechoslovak Math.\,J. \textbf{65}(140) (2015), no.\,1,
\hbox{21\hs--59}.

\bibitem{besse}A.\,L.\,Besse, \textit{Einstein Manifolds}, Ergebnisse der
Mathematik und ihrer Grenzgebiete (3), Springer-Verlag, Berlin, 1987.

\bibitem{caeiro-oliveira-marino-villar}S.\,Caeiro-Oliveira and
R.\,Mari\~no-Villar, \textit{An algebraic characterization of
weak\-ly-Ein\-stein hyper\-sur\-faces in space forms}, 
J.\,Geom.\,Phys. \textbf{214} (2025), art.\,105530, 10 pp.

\bibitem{derdzinski-00} A.\,Derdzinski, \textit{Ein\-stein metrics in
dimension four}, Handbook of Differential Geometry, vol. I,
pp.\,\hbox{419\hs--707}. North-Holland, Amsterdam (2000).

\bibitem{derdzinski-euh-kim-park} A.\,Derdzinski, Y\nnh.\,Euh, S.\,Kim and
J.\,H.\,Park, \textit{On weakly Ein\-stein K\"ah\-ler surfaces}, 
In\-ter\-nat.\,J.\,Math. \textbf{36} (2025), no.\,14, art.\,2550057, 23 pp.

\bibitem{derdzinski-park-shin-ct} 
A.\,Derdzinski, J.\,H.\,Park, and W\nnh.\,Shin, {\it Weakly Einstein
curvature tensors}, preprint, available from 
\hbox{https:/\hskip-1.6pt/arxiv.org/pdf/2504.18752.}

\bibitem{derdzinski-park-shin-cp}A.\,Derdzinski, J.\,H.\,Park and
W\nnh.\,Shin, \textit{Weakly Ein\-stein con\-for\-mal products}, preprint,
available from https:/\hskip-1.7pt/arxiv.org/abs/2512.05173.

\bibitem{derdzinski-piccione-terek}A.\,Derdzinski,
P\nnh.\,Piccione and I.\,Terek, 
{\em Nijen\-huis geometry of parallel tensors}, Ann.\,Mat. Pu\-ra Appl. (4)
\textbf{204} (2025), no.\,4, 
1381--1401.

\bibitem{euh-park-sekigawa-rm}
Y\nnh.\,Euh, J.\,H.\,Park and K.\,Sekigawa, \textit{A curvature identity on a
$\,4$-dimensional Riemannian manifold}, Results Math. \textbf{63} (2013),
no.\,1-2, 107--114. 

\bibitem{euh-park-sekigawa-ms}
Y.\,Euh, J.\,H.\,Park and K.\,Sekigawa, \textit{A generalization of a $4$-dimensional Ein\-stein manifold}, Math. Slovaca \textbf{63} (2013), 
\hbox{595\hs--610}.

\bibitem{garcia-rio-haji-badali-marino-villar-vazquez-abal}
E.\,Garc\'\i a-R\'\i o, A.\,Haji-Badali, R.\,Mari\~no-Villar and
M.\,E.\,V\'azquez-Abal, \textit{Locally con\-for\-mal\-ly flat \
weakly-Ein\-stein manifolds}, Arch.\,Math.\ (Basel) \textbf{111} (2018),
no.\,5, \hbox{549\hs--559}.

\bibitem{kim-lebrun-pontecorvo}J.\,Kim, C.\,LeBrun and M.\,Pontecorvo, 
\textit{Sca\-lar-flat K\"ah\-ler surfaces of all ge\-ne\-ra},
J.\,reine angew.\,Math.\ \textbf{486} (1997), 69--95.

\bibitem{lebrun}C.\,LeBrun, \textit{Sca\-lar-flat K\"ah\-ler metrics on
blown-Up ruled surfaces}, J.\,reine angew.\,Math.\ \textbf{420} (1991),
161--177.

\bibitem{marino-villar}R.\,Mari\~no-Villar, \textit{Structure of locally
con\-for\-mal\-ly flat manifolds satisfying some weak\-ly-Ein\-stein
conditions}, J.\,Geom\,Phys. \textbf{186} (2023), art.\,104754, 8 pp.

\bibitem{tanno}
S.\,Tanno, \textit{$4$-dimensional conformally flat Kahler manifolds},
T\^ohoku Math.\,J. \textbf{24} (1972), no.\,3, 
\hbox{501\hs--504}.

\bibitem{wang-zhang}Y\nnh.\,Wang and Y\nnh.\,Zhang, \textit{Weakly Ein\-stein
real hyper\-sur\-faces in\/ $\mathbb CP^2$ and\/ $\mathbb CH^2\nh$,} 
J.\,Geom.\,Phys.\ \textbf{181} (2022), art.\,104648, 11 pp.

\end{thebibliography}
\end{document}